\documentclass[a4paper,11pt,oneside]{amsart}
\usepackage{thomas}
\usepackage{xnotes}
\hypersetup{hidelinks,final,unicode}

\title{Coarse injectivity, hierarchical hyperbolicity, and semihyperbolicity}

\author{Thomas Haettel}
\address{Institut Montpelli\'{e}rain Alexander Grothendieck, CNRS, Univ. Montpellier. \\
Place Eugène Bataillon, 34090 Montpellier, France}
\email{thomas.haettel@umontpellier.fr}

\author{Nima Hoda}
\address{DMA, École normale supérieure \\ Université
  PSL, CNRS \\ 75005 Paris, France}
\email{nima.hoda@mail.mcgill.ca}

\author{Harry Petyt}
\address{School of Mathematics, University of Bristol, UK}
\email{h.petyt@bristol.ac.uk}

\date{\today}


\makeatletter
\@namedef{subjclassname@2020}{\textup{2020} Mathematics Subject Classification}
\makeatother


\begin{document}

\begin{abstract}
    We relate three classes of nonpositively curved metric spaces: hierarchically hyperbolic spaces, coarsely injective spaces, and strongly shortcut spaces.  We show that every hierarchically hyperbolic space admits a new metric that is coarsely injective.  The new metric is quasi-isometric to the original metric and is preserved under automorphisms of the hierarchically hyperbolic space. We show that every coarsely injective metric space of uniformly bounded geometry is strongly shortcut.  Consequently, hierarchically hyperbolic groups---including mapping class groups of surfaces---are coarsely injective and coarsely injective groups are strongly shortcut.  
    
    Using these results, we deduce several important properties of hierarchically hyperbolic groups, including that they are semihyperbolic, have solvable conjugacy problem, have finitely many conjugacy classes of finite subgroups, and that their finitely generated abelian subgroups are undistorted.  Along the way we show that hierarchically quasiconvex subgroups of hierarchically hyperbolic groups have bounded packing.
\end{abstract}

\maketitle

\section{Introduction}

A principal theme of geometric group theory is the study of groups as metric spaces.  This includes studying groups via the types of metric spaces they act on.  In this vein, the study of groups acting on spaces satisfying various forms of nonpositive curvature conditions has been especially fruitful.  In this article, we are concerned with three classes of spaces exhibiting nonpositive curvature: hierarchically hyperbolic spaces, coarsely injective spaces, and strongly shortcut spaces.

\subsection{The setting} ~

The first of our three classes is that of \emph{hierarchically hyperbolic spaces}, which were introduced by Behrstock, Hagen, and Sisto in \cite{behrstockhagensisto:hierarchically:1}. These spaces exhibit hyperbolic-like behaviour, and there is a growing body of interesting examples, including many quotients of mapping class groups \cite{behrstockhagenmartinsisto:combinatorial} and all known cubical groups \cite{hagensusse:onhierarchical}, amongst others. The theory has had a number of successes, such as proving Farb's quasiflats conjecture for mapping class groups \cite{hhs_quasiflats} and establishing uniform exponential growth for many cubical groups \cite{abbottngspriano:hierarchically}.  We postpone describing the hierarchy structure until Section~\ref{subsection:hhsbackground}.

\mk

The next class we consider is that of coarsely injective spaces. A metric space is said to be \emph{coarsely injective} if there is
a constant $\delta$ such that for any family $\{B(x_i,r_i):i\in I\}$
of balls with $d(x_i,x_j)\leq r_i+r_j$ for all $i,j\in I$, the
$\delta$--neighbourhoods of those balls have nonempty total
intersection. This property was first considered by Chepoi and Estellon in
\cite{chepoiestellon:packing}.

As the term suggests, this notion is closely related to that of \emph{injective} metric spaces. A metric space is injective (also called \emph{hyperconvex}) if for any
family $\{B(x_i,r_i):i\in I\}$ of balls with
$d(x_i,x_j)\leq r_i+r_j$ for all $i,j\in I$, the balls have
nonempty total intersection. In other words, it amounts to taking $\delta=0$ in the definition of coarse injectivity. (There are multiple equivalent ways to define injectivity of a metric space, by a theorem of Aronszajn and Panitchpakdi
\cite{aronszajnpanitchpakdi:extension}.) A construction of Isbell
\cite{isbell:six}, which was later rediscovered by Dress
\cite{dress:trees} and by Chrobak and Larmore
\cite{chrobaklarmore:generosity}, shows that every metric space has an
essentially unique \emph{injective hull}. More precisely, the
injective hull of a metric space $X$ is an injective metric space
$E(X)$, together with an isometric embedding $e:X\to E(X)$, such that
no injective proper subspace of $E(X)$ contains $e(X)$.  For convenience, we will often identify $X$ with its image $e(X)$.  A nice
description of the construction of the injective hull is given by Lang in
\cite[\S3]{lang:injective}. 

The classes of coarsely injective spaces and injective spaces are tied together by the following useful fact, the proof of which is identical to that of \cite[Proposition~3.12]{chalopinchepoigenevoishiraiosajda:helly}. A subset $Y$ of a metric space $X$ is \emph{coarsely dense} if there exists $r$ such that every $x\in X$ is $r$--close to some $y\in Y$.
\bpro \label{prop:dense_in_hull}
A metric space is coarsely injective if and only if it is coarsely dense in its injective hull.
\epro
Moreover, if a group acts properly and coboundedly on a coarsely injective space, then it acts properly and coboundedly on the injective hull of that space (see Lemma~\ref{lem:modifiedccgho}).  Here and throughout the paper, a group $G$ is said to act \emph{properly} on a metric space $X$ if $\{g \in G : gB \cap B \neq \emptyset \}$ is finite for every metric ball $B$ of $X$.  This is sometimes referred to as a \emph{metrically proper} action.

Injective metric spaces satisfy a number of properties reminiscent of nonpositive curvature, and in particular of CAT(0) spaces. For instance, they admit a conical geodesic bicombing \cite{lang:injective}, and proper injective spaces of finite combinatorial dimension have a canonical convex such bicombing \cite{descombeslang:convex}. Also, every bounded group action on an injective metric space has a fixed point, and the fixed point set is itself injective \cite{lang:injective}. These properties are what allow us to draw our conclusions for hierarchically hyperbolic groups. Although it will not be needed here, it is interesting to note that injective spaces are also complete \cite{aronszajnpanitchpakdi:extension} and contractible \cite{isbell:six}.

\mk

The strong shortcut property was introduced by the second named author
for graphs \cite{Hoda:shortcut_graphs} and then generalized to roughly
geodesic metric spaces \cite{Hoda:shortcut_spaces}.  A \emph{Riemannian circle} $S$ is the circle $S^1$ endowed with a geodesic metric of some length $|S|$.  A roughly geodesic metric space $(X,\sigma)$ is \emph{strongly shortcut} if there exists $K > 1$ such that for any $C > 0$ there is a bound on the lengths $|S|$ of $(K,C)$--quasi-isometric embeddings $S \to X$ of Riemannian circles $S$ in $(X,\sigma)$.  A group is \emph{strongly shortcut} if it acts properly and coboundedly on a strongly shortcut metric space.  Many spaces and
graphs of interest in geometric group theory and metric graph theory
are strongly shortcut, including Gromov-hyperbolic spaces, $1$--skeletons of finite-dimensional
$\CAT(0)$ cube complexes, Cayley graphs of Coxeter groups, and asymptotically $\CAT(0)$ spaces.  Despite being such a unifying notion, it
remains possible to draw conclusions about strongly shortcut groups, including that they are finitely presented and have polynomial isoperimetric function, and so have decidable word problem.

\subsection{Comparison of the classes} ~

Our main result is the definition of a new metric on hierarchically hyperbolic spaces, and more generally on coarse median spaces satisfying a nice approximation property of median intervals by CAT(0) cube complexes.

Our construction is directly inspired by work of Bowditch (see~\cite{bowditch_median_injective}), in which he constructs an injective metric on any finite rank metric median space. Indeed, if one endows a finite-dimensional CAT(0) cube complex with the piecewise $\ell^\infty$ metric, it becomes an injective metric space.  The new metric we construct is weakly roughly geodesic and has the property that balls are coarsely median-convex; see Theorem~\ref{thm:properties_of_sigma}.

We then prove a hierarchical generalisation of a very nice result of Chepoi, Dragan, and Vax\`es \cite{chepoidraganvaxes:core} about pairwise close subsets of hyperbolic spaces. Combining this with work of Russell, Spriano, and Tran \cite{hhs_convexity} enables us to deduce a coarse Helly property for balls.

\bmthm[Proposition~\ref{prop:bilipschitzmetrics}, Corollary~\ref{cor:hhscoarselyhelly}] \label{thm:hhs_coarsely_helly}
Let $(X,\s)$ be a hierarchically hyperbolic space with metric $d$. There exists a metric $\sigma$ on $X$ such that $(X,\sigma)$ is coarsely injective and quasi-isometric to $(X,d)$. Moreover, $\sigma$ is invariant with respect to the automorphism group of $(X,\s)$.
\emthm

Our second result relates the class of coarsely injective spaces to that of strongly shortcut spaces.  A metric space has \emph{uniformly bounded geometry} if, for any $r > 0$, there exists a uniform $N(r) \in \N$ such that every ball of radius $r$ contains at most $N(r)$ points.

\bmthm[Theorem~\ref{thm:coarse_helly_implies_shortcut}] \label{mthm:coarsely_helly_shortcut}
Every coarsely injective metric space of uniformly bounded geometry is strongly shortcut.
\emthm

Huang and Osajda proved that weak Garside groups of finite type and Artin groups of $FC$-type are Helly \cite{huang:garside}, so we have the following corollary of Theorem~\ref{mthm:coarsely_helly_shortcut}.

\bmcor
    Weak Garside groups of finite type and Artin groups of $FC$-type are strongly shortcut.
\emcor

Combining Theorem~\ref{thm:hhs_coarsely_helly} with Theorem~\ref{mthm:coarsely_helly_shortcut}, we deduce the following.

\bmcor
Every hierarchically hyperbolic space admits a roughly geodesic metric in its quasi-isometry class that satisfies the strong shortcut property.
\emcor

In fact, in the case of hierarchically hyperbolic \emph{groups}, the metric we construct is equivariant, by the ``moreover'' statement of Theorem~\ref{thm:hhs_coarsely_helly} (also see Remark~\ref{remark:HHS_aut_is_enough}). We therefore have that every hierarchically hyperbolic group acts properly cocompactly on a coarsely injective space, and any group admitting such an action is a strongly shortcut group. Moreover, these three classes can be distinguished. Indeed, the $(3,3,3)$ Coxeter triangle group is strongly shortcut but not coarsely injective \cite{hoda:crystallographic}; and type-preserving uniform lattices in thick buildings of type $C_n$ are coarsely injective \cite{chalopinchepoigenevoishiraiosajda:helly}, but they cannot be hierarchically hyperbolic groups, because they do not admit any nonelementary actions on hyperbolic spaces \cite{haettel:hyperbolic_rigidity}.

\mk

One can also ask how \emph{Helly groups}, as defined in \cite{chalopinchepoigenevoishiraiosajda:helly}, fit into this framework. A \emph{Helly graph} is a locally finite graph in which any set of pairwise intersecting balls in the vertex set have nonempty total intersection, and a group is Helly if it acts properly cocompactly on a Helly graph. Helly groups have some strong properties, including biautomaticity \cite[Theorem~1.5]{chalopinchepoigenevoishiraiosajda:helly}. Hughes and Valiunas have constructed a hierarchically hyperbolic group that is not biautomatic and hence not Helly, though it is CAT(0) and acts properly cocompactly on an injective metric space \cite{hughesvaliunas:commensurating}.

It is clear that every Helly group is coarsely injective. Recall that according to Bridson (see~\cite{bridson:semisimple}), mapping class groups are not CAT(0). Note that any Helly group acts properly cocompactly on a space with a convex geodesic bicombing (see~\cite{descombeslang:convex}). So we suspect that mapping class groups are not Helly groups.

We can summarise the relations between these classes with the following diagram, in which $A \Rightarrow B$ denotes the statement ``any group that is $A$ is necessarily $B$'', and $A \not\Rightarrow B$ denotes ``there is an example of a group that is $A$ but not $B$''.

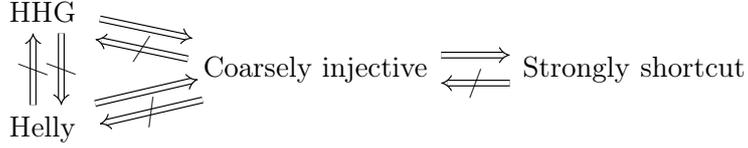
\[\begin{tikzcd}[row sep = tiny]
\text{HHG } \arrow[dr, shift left, Rightarrow] \arrow[dd, shift left, shift left, Rightarrow, "/" marking] & & \\ 
    & \hspace{3mm}\text{Coarsely injective} \arrow[r, shift left, shift left, Rightarrow] \arrow[ul, shift left, shift left, Rightarrow, "/" marking] \arrow[dl, shift left, Rightarrow, "/" marking]
        & \text{Strongly shortcut} \arrow[l, shift left, shift left, Rightarrow, "/" marking]\\ 
    \text{Helly } \arrow[uu, shift left, shift left, Rightarrow, "/" marking] \arrow[ur, shift left, shift left, Rightarrow] & &
\end{tikzcd}\]

\subsection{Metric consequences} \label{subsection:metric_consequences}

We now describe some of the consequences of Theorem~\ref{thm:hhs_coarsely_helly} for hierarchically hyperbolic spaces. Recall that a quasigeodesic bicombing on a metric space $(X,\sigma)$ is a map $\gamma : X \times X \times [0,1] \ra X$ such that, for each distinct pair $a,b \in X$, the map $[0,\sigma(a,b)] \to X$ given by $t \mapsto \gamma_{a,b}\left(\tfrac{t}{\sigma(a,b)}\right)$ is a quasigeodesic from $a$ to $b$ with uniform constants.

There are various fellow-travelling conditions that a bicombing may enjoy. We say that a bicombing is \emph{roughly conical} if
\[\exists C \geq 0, \forall a,b,a',b' \in X, \forall t \in [0,1], \sigma(\gamma_{a,b}(t),\gamma_{a',b'}(t)) \leq (1-t)\sigma(a,a')+t\sigma(b,b')+C,\]
and it is \emph{roughly reversible} if it satisfies the following coarse version of symmetry:
\[ \exists C \ge 0, \forall a,b \in X, \forall t \in [0,1], \sigma(\gamma_{a,b}(t),\gamma_{b,a}(1-t)) \le C. \]

From the existence of conical, reversible, isometry-invariant geodesic bicombings on injective metric spaces \cite{lang:injective}, we deduce the following. 

\bmcor[Corollary~\ref{cor:hhs_bicombing}] \label{cor:hhs_bicombing_main}
Let $(X,\s)$ be a hierarchically hyperbolic space. Then $(X,\sigma)$ admits a roughly conical, roughly reversible, quasigeodesic bicombing that is coarsely equivariant under the automorphism group of $(X,\s)$. More strongly, the combing lines are rough geodesics for the metric $\sigma$.
\emcor

In particular, this applies to Teichm\"uller space with either of the standard metrics, with equivariance under the action of the mapping class group.  This particular application was unknown to us until comparing results with Durham, Minsky, and Sisto \cite{durham_minsky_sisto:semihyperbolic}.

\mk

Corollary~\ref{cor:hhs_bicombing_main} gives a positive answer to Question~8.1 of~\cite{engelwulff:coronas}, as any roughly conical bicombing is \emph{coherent and expanding}, in the terminology of~\cite{engelwulff:coronas}. Engel and Wulff proved that the existence of such a bicombing has a large number of $K$--theoretic consequences. This positive answer also allows one to apply work of Fukaya--Oguni (see \cite{fukayaoguni:cartanhadamard}) to deduce the coarse Baum--Connes conjecture for hierarchically hyperbolic groups. The coarse Baum--Connes conjecture is also a consequence of finite asymptotic dimension, which is a known property of uniformly proper hierarchically hyperbolic spaces \cite{behrstockhagensisto:asymptotic}.

\subsection{Consequences for groups} ~

We now turn to the case of hierarchically hyperbolic groups, which, as we have seen, act properly cocompactly on coarsely injective spaces. Here we describe some of the consequences of such an action.

\mk

Following Alonso and Bridson \cite{alonsobridson:semihyperbolic}, we say that a bicombing is \emph{bounded} if it satisfies the following weak two-sided fellow-traveller property:
\[\exists C \geq 0, \forall a,b,a',b' \in X, \forall t \in [0,1], \sigma(\gamma_{a,b}(t),\gamma_{a',b'}(t)) \leq C \max(\sigma(a,a'),\sigma(b,b'))+C.\]
Note that if a bicombing is roughly conical, then it is bounded. A finitely generated group is said to be \emph{semihyperbolic} if it has a Cayley graph that admits an equivariant bounded quasigeodesic bicombing. 

Among other results, Alonso and Bridson proved that semihyperbolicity implies the existence of a quadratic isoperimetric function, that the group has soluble word and conjugacy problems, and that an algebraic flat torus theorem holds \cite{alonsobridson:semihyperbolic}.  For more discussion of the consequences of semihyperbolicity we refer the reader to \cite{bridsonhaefliger:metric}.  Semihyperbolicity was introduced as a response to Gromov's call for a weaker form of hyperbolicity in his original essay on hyperbolic groups, and it fits into the framework of algorithmic properties developed in \cite{epsteincannonholtlevypatersonthurston:word}. For example, semihyperbolicity is implied by biautomaticity, but not by automaticity. A survey can be found in \cite{bridson:semihyperbolicity}.

For hierarchically hyperbolic groups $G$, the freeness of the regular action of $G$ on $(G,\sigma)$ allows the bicombing of Corollary~\ref{cor:hhs_bicombing_main} to be pulled back to the Cayley graph of $G$ \cite{alonsobridson:semihyperbolic}.

\bmcor[Corollary~\ref{semihyperbolicity}]
Every hierarchically hyperbolic group is semihyperbolic. In particular, the mapping class group of a surface of finite type is semihyperbolic.
\emcor

The mapping class group case is also a consequence of unpublished work of Hamenstädt \cite{hamenstadt:biautomatic}, and is related to Mosher's automaticity theorem \cite{mosher:mapping}.

We should emphasise that the same result for mapping class groups has been obtained by rather different methods, simultaneously and independently, by Durham, Minsky, and Sisto (see~\cite{durham_minsky_sisto:semihyperbolic}). This will be discussed more in Section~\ref{subsection:comparison}.

\mk

It is well-known that mapping class groups have finitely many conjugacy classes of finite subgroups (see \cite{bridson:finiteness}), a property that they share with hyperbolic groups. However, to the authors' knowledge, all existing proofs of this fact rely on deep results that do not generalise to other settings, such as Kerckhoff's celebrated solution of the Nielsen realisation problem \cite{kerckhoff:nielsen}. It is interesting to ask whether there is a proof that avoids such powerful machinery, and indeed a more general question about hierarchically hyperbolic groups was asked in \cite{hagenpetyt:projection}. The question of whether all hierarchically hyperbolic groups have finitely many conjugacy classes of finite subgroups has resisted a number of attempted resolutions. 

The fact that hierarchically hyperbolic groups act properly cocompactly on coarsely injective spaces makes the following a simple consequence of Lang's result about bounded actions on injective spaces \cite[Proposition~1.2]{lang:injective}.

\bmthm[Corollary~\ref{prop:torsion}]
Hierarchically hyperbolic groups have finitely many conjugacy classes of finite subgroups.
\emthm

It is interesting to note that this applies in particular to many quotients of mapping class groups \cite{behrstockhagensisto:asymptotic,behrstockhagenmartinsisto:combinatorial}. It is also a simple consequence that residually finite hierarchically hyperbolic groups are virtually torsionfree.

\mk

We now summarise the consequences for hierarchically hyperbolic groups of the results described above (also see Remark~\ref{remark:HHS_aut_is_enough} for a comment on their generality).

\bmcor \label{cor:list}
Every hierarchically hyperbolic group $G$ has the following properties. 
\begin{shortitem}
\item $G$ acts properly cocompactly on a proper coarsely injective space.
\item $G$ has finitely many conjugacy classes of finite subgroups.
\item $G$ is semihyperbolic.  In particular:
\begin{shortitem}
\item the conjugacy problem in $G$ is soluble, and it can be solved in doubly exponential time;
\item any polycyclic subgroup of $G$ is virtually abelian;
\item any finitely generated abelian subgroup of $G$ is quasi-isometrically embedded;
\item the centraliser of any finite subset of $G$ is finitely generated, quasi-isometrically embedded, and semihyperbolic.
\end{shortitem}
\item For any ring $R$, if the cohomological dimension $cd_R(G)$ is finite, then $cd_R(G) \leq asdim(G)+1$ \cite[Theorem~C]{engelwulff:coronas}.
\item $G$ is a strongly shortcut group.
\end{shortitem}
\emcor

The result about polycyclic subgroups can also be deduced from the Tits alternative for hierarchically hyperbolic groups established in \cite{durhamhagensisto:boundaries}. The result about finitely generated abelian subgroups was proved by Plummer \cite{plummer:abelian}.  The other consequences are new, however. The result about the conjugacy problem extends work of Abbott and Behrstock showing that it can be solved in exponential time for \emph{Morse elements} of hierarchically hyperbolic groups \cite{abbottbehrstock:conjugator}, and generalises the fact that, in mapping class groups, it can always be solved in exponential time \cite{masurminsky:geometry:2,tao:linearly,behrstockdrutu:divergence}. In the case of cubical groups, a beautiful result of Niblo and Reeves states that every cubical group is biautomatic \cite{nibloreeves:geometry}, and semihyperbolicity is a direct consequence of this. We emphasise, though, that the class of hierarchically hyperbolic groups is considerably larger than just cubical groups and mapping class groups.

\subsection{Bounded packing} ~

The bounded packing property for subgroups of finitely generated groups was introduced as a metric abstraction of tools used by several authors to prove intersection properties of subgroups of hyperbolic groups \cite{gitikmitraripssageev:widths,rubinsteinsageev:intersection}, and in turn as a stepping stone towards ensuring cocompactness of the cube complex associated with a finite collection of quasiconvex codimension--1 subgroups \cite{sageev:codimension,nibloreeves:coxeter,hruskawise:finiteness}. We recall the definition in Section~\ref{section:hhs}; see \cite{hruskawise:packing,hruskawise:finiteness} for more motivation and background. The prototypical example is that of a quasiconvex subgroup of a hyperbolic group. That such subgroups have bounded packing was first established by Gitik, Mitra, Rips, and Sageev, using compactness of the boundary \cite{gitikmitraripssageev:widths}, and another proof was given by Hruska--Wise, using induction on the \emph{height} of the subgroups \cite{hruskawise:packing}. 

More general examples have been provided by Antol\'in, Mj, Sisto, and Taylor, who use induction on height to show that finite collections of \emph{stable subgroups} in any finitely generated group have bounded packing \cite{antolinmjsistotaylor:intersection}. Stable subgroups were introduced by Durham and Taylor \cite{durhamtaylor:convex}, and they are always hyperbolic. More generally, \emph{Morse subgroups} were introduced independently by Tran \cite{tran:onstrongly} and Genevois \cite{genevois:hyperbolicities}, and the notion is implicit in earlier work of Sisto \cite{sisto:hyperbolically_embedded}. Notably, Tran proved that any finite collection of Morse subgroups has bounded packing \cite[Theorem~1.2]{tran:onstrongly}, again by using induction on height. 

\bmthm[Corollary~\ref{cor:boundedpacking}] \label{mthm:packing}
Every finite collection of hierarchically quasiconvex subgroups of a group that is a hierarchically hyperbolic space (in particular, of any hierarchically hyperbolic group) has bounded packing.
\emthm

For many groups that are HHSs (including all HHGs), every stable subgroup is hierarchically quasiconvex \cite{abbottbehrstockdurham:largest,hhs_convexity}. Theorem~\ref{mthm:coarsely_helly_shortcut} also applies to subsurface stabilisers in the mapping class group, which are neither Morse nor stable. See Section~\ref{subsection:hhsbackground} for the definition of hierarchical quasiconvexity. 

Our proof of this result is purely geometric. It relies on a very strong result for quasiconvex subsets of hyperbolic spaces that was proved by Chepoi, Dragan, and Vax\`es \cite{chepoidraganvaxes:core}; we state it as Theorem~\ref{thm:cdvhelly}. Their theorem does not seem to have garnered the notice it deserves in geometric group theory. For instance, it yields what appears to be the simplest and most natural proof of bounded packing for quasiconvex subgroups of hyperbolic groups. One case of our hierarchical generalisation of their result can be stated as follows.

\bmthm[Theorem~\ref{thm:hellyhqc}]
Let $X$ be a hierarchically hyperbolic space, and let $\mathcal Q$ be a finite collection of hierarchically quasiconvex subsets of $X$. If every pair of elements of $\mathcal Q$ is $r$--close, then there is a point of $X$ that is $R$--close to every element of $\mathcal Q$, where $R$ does not depend on the cardinality of $\mathcal Q$.
\emthm

\subsection{Comparison to the work of Durham, Minsky, and Sisto {\cite{durham_minsky_sisto:semihyperbolic}}} \label{subsection:comparison} ~

Let us now say a few words about the difference between the present article and the work of Durham, Minsky, and Sisto \cite{durham_minsky_sisto:semihyperbolic}. As noted, both articles independently prove that mapping class groups are semihyperbolic, but the approaches differ greatly. In both cases, this fact is deduced from a stronger statement in a more general setting, but those two statements are very different in flavour. Their results hold for hierarchically hyperbolic spaces with the extra assumption of \emph{colorability}, and they deduce interesting corollaries about bicombings on the Teichmüller space with the Teichmüller metric, and the existence of barycentres. These results are also consequences of Theorem~\ref{thm:hhs_coarsely_helly} and Corollary~\ref{cor:hhs_bicombing_main}.

Our construction is built on the fact that intervals in hierarchically hyperbolic spaces can be approximated by finite CAT(0) cube complexes (proved in~\cite{hhs_quasiflats}). The main result of Durham, Minsky, and Sisto is that these approximations are furthermore \emph{stable}, meaning that a small change in the endpoints of the interval induces a small change in the approximating CAT(0) cube complex. This stability result may prove extremely useful for other purposes.

If we want to compare the bicombing we obtain to the one from \cite{durham_minsky_sisto:semihyperbolic} in the simplest case of a CAT(0) cube complex, our bicombing looks like the geodesic CAT(0) bicombing, whereas their bicombing is more similar to (but not the same as) Niblo--Reeves normal cube paths \cite{nibloreeves:geometry}. One notable difference is that our bicombing is roughly conical and their bicombing is merely bounded, which is not enough to deduce the consequences of Section~\ref{subsection:metric_consequences}. On the other hand, their bicombing paths are known to be hierarchy paths, whilst ours are not.

\subsection{Structure of the article} ~

In Section~\ref{section:newmetric}, we recall basic definitions of coarse median spaces, and we explain the extra property we need, a stronger approximation of median intervals by CAT(0) cube complexes. We then define a new distance, and we prove that it is quasi-isometric to the original one, is weakly roughly geodesic, and its balls are coarsely median-convex.

In Section~\ref{section:hhs}, we treat hierarchical hyperbolicity, and prove that hierarchically quasiconvex subsets satisfy a coarse version of the Helly property. We use this to show that the new distance makes hierarchically hyperbolic spaces coarsely injective, and deduce semihyperbolicity of hierarchically hyperbolic groups. We also show that hierarchically quasiconvex subgroups have bounded packing.

In Section~\ref{section:shortcut}, we recall the definition of a strongly shortcut group, and prove that coarse injectivity implies the strong shortcut property.

\mk

{\bf Acknowledgements:} We thank Mark Hagen and Anthony Genevois for many helpful comments and suggestions, and Victor Chepoi, Alexander Engel and Damian Osajda for interesting discussions. We would like to thank Matthew Durham, Yair Minsky, and Alessandro Sisto for friendly discussions about our their work and ours. We thank the anonymous referees for their careful and thorough readings of the article.

The first named author was supported by the french grant ANR-16-CE40-0022-01 AGIRA.

The second named author was supported by the ERC grant GroIsRan. 

\section{Coarse median spaces with quasicubical intervals} \label{section:newmetric}

\subsection{Background on coarse median spaces} ~

\emph{Coarse median spaces}, defined by Bowditch in~\cite{bowditch_coarse_median}, are a generalisation of CAT(0) cube complexes and Gromov-hyperbolic spaces, and the class is rich enough to encompass mapping class groups of finite type surfaces. The general idea is to associate to every triple of points in the space a point that satisfies the axioms of a usual median up to controlled error. This point will be called the coarse median.

\mk

Let us recall here that a \emph{median} $\mu : X^3 \ra X$ on a set $X$ is a map satisfying (where we write equivalently $\mu(x,y,z)$ or $\mu_{x,y,z}$ to increase readability):
\begin{shortitem}
\item $\mu(x,y,z)$ is symmetric in $x,y,z$,
\item $\forall x,y \in X, \mu(x,x,y)=x$ and
\item $\forall a,b,x,y,z \in X, \mu(a,b,\mu_{x,y,z}) = \mu(\mu_{a,b,x},\mu_{a,b,y},z).$
\end{shortitem}
The pair $(X,\mu)$ is called a \emph{median algebra}. The \emph{rank} of $(X,\mu)$ is the supremum of all $\nu \in \N$ such that there exists an injective median homomorphism from the $\nu$--cube $\{0,1\}^\nu$ into $X$. 

\mk

Every finite median algebra can be seen as the $0$-skeleton of a CAT(0) cube complex (see~\cite{chepoi:graphs,roller:poc}).

Let $(X,d)$ be a metric space. For any $x,y \in X$, let $I_d(x,y) = \{z \in X \st d(x,z)+d(z,y)=d(x,y)\}$ denote the interval between $x$ and $y$. The metric space $(X,d)$ is called
\emph{metric median} if $\forall x,y,z \in X, I_d(x,y) \cap I_d(y,z) \cap
I_d(x,z)$ is a singleton, say $\{\mu(x,y,z)\}$. In this case, $\mu$ defines
a median on $X$. Examples of median metric spaces include trees, $1$--skeletons
of CAT(0) cube complexes with the combinatorial distance, and
$L^1$ spaces.

\mk

In a Gromov-hyperbolic space $X$, the three intervals joining three points may not intersect precisely in a singleton, but by definition they do coarsely intersect with uniformly bounded diameter. This suggests defining a map $X^3 \ra X$ that satisfies the axioms of a median up to bounded error. This is made precise by the following definition due to Bowditch \cite{bowditch_coarse_median}, generalising the \emph{centroid} defined for mapping class groups in \cite{behrstockminsky:centroids}.

\bdf[Coarse median space] \label{def:coarse_median} Let $(X,d)$ be a metric space. A map $\mu:X^3 \ra X$ is called a \emph{coarse median} if there exists $h:\N \ra (0,\pif)$ such that
\begin{shortitem}
\item For all $a,b,c,a',b',c' \in X$, we have $d(\mu(a,b,c),\mu(a',b',c')) \leq h(0)(d(a,a')+d(b,b')+d(c,c')) + h(0)$.
\item For each finite non-empty set $A \subset X$, with $|A| \leq n$, there exists a finite median algebra $(Q,\mu_Q)$ and maps $\pi:A \ra Q$, $\lambda:Q \ra X$ such that for every $\alpha,\beta,\gamma \in Q$, we have $d(\lambda \mu_Q(\alpha,\beta,\gamma),\mu(\lambda \alpha,\lambda \beta,\lambda \gamma)) \leq h(n)$, and for every $a \in A$, we have $d(a,\lambda(\pi(a))) \leq h(n)$.
\end{shortitem}
We say that the triple $(X,\mu,d)$ is a \emph{coarse median space}. If $Q$ can always be chosen to have rank at most $\nu$, we say that $\mu$ has rank at most $\nu$. As with median algebras, we shall write $\mu_{a,b,c}=\mu(a,b,c)$ interchangeably. Note that we are also free to assume that $\mu(a,b,c)$ is symmetric in $a,b,c$, and that $\mu(a,a,b)=a$ \cite[p.73]{bowditch_coarse_median}.
\edf

We now recall the definitions of intervals and coarse convexity in coarse median spaces.

\bdf[Median interval]
For a pair of points $a,b \in X$, the median interval between $a$ and $b$ is defined as
\[ [a,b] = \{\mu(a,b,x) \st x \in X\}.\]
\edf

\bdf[Coarse median-convexity]
For a constant $M\ge0$, a subset $Y$ of $X$ is said to be $M$--coarsely median-convex if \[d(Y,\mu(x,y,y'))\le M \text{ for all }y,y'\in Y,\text{ }x\in X.\]
\edf

We finish by introducing some terminology.

\bdf[Weakly roughly geodesic] \label{def:wrg}
Recall that a metric space $(X,d)$ (or, more briefly, the metric $d$) is called \emph{roughly geodesic} if there exists a constant $C_d \geq 0$ such that, for any $a,b\in X$, there exists a $(1,C_d)$--quasi-isometric embedding of the interval $f:[0,d(a,b)] \ra X$ such that $f(0)=a$ and $f(d(a,b))=b$. 

We say that a metric space $(X,d)$ is called \emph{weakly roughly geodesic} if there exists a constant $C'_d \geq 0$ such that, for any $a,b\in X$ and any nonnegative $r\le d(a,b)$, there is a point $c\in X$ with $|d(a,c)-r|\le C'_d$ and $d(a,c)+d(c,b)\le d(a,b)+C'_d$. 
\edf

\begin{rmk}
Note that every roughly geodesic metric space is weakly roughly geodesic. Moreover, any metric space $(X,d)$ that is weakly roughly geodesic with constant $C'_d$ is necessarily $(4C'_d,4C'_d)$--quasigeodesic. Indeed, given $x,y\in X$, one can repeatedly take $r=3C'_d$ in the definition of weak rough geodesicity to get a sequence $x=w_0,w_1,\dots,w_n=y$ such that $d(w_i,w_{i+1})\in[2C_d',4C_d']$ and $d(w_i,y)\le d(w_{i-1},y)-C_d'$, and the points of this sequence form a quasigeodesic from $x$ to $y$.
\end{rmk}

\subsection{Construction of a new metric} ~

Let $(X, \mu, d)$ be a coarse median space. Following Bowditch's construction of an injective metric on a median metric space \cite{bowditch_median_injective}, we shall define a new metric $\sigma$ on~$X$.

\bdf[Contraction] For a constant $K \geq 0$, a map $\Phi : X \ra \R$ is called a $K$--contraction if:
\bit
\item $\Phi$ is $(1,K)$--coarsely Lipschitz, i.e. $\forall a,b \in X, |\Phi(x)-\Phi(y)| \leq d(x,y)+K$. 
\item $\Phi$ is a $K$--quasi-median homomorphism, i.e. \[\forall a,b,c \in X, |\Phi(\mu(a,b,c)) - \mu_\R(\Phi(a),\Phi(b),\Phi(c))| \leq K,\] where $\mu_\R$ denotes the standard median on $\R$.
\eit
\edf

\bdf[New metric]
For $K > 0$, we define a new metric $\sigma$ on $X$ as follows. Given $a,b \in X$, let $\sigma(a,b)$ denote the supremum of all $r \geq 0$ such that there exists a $K$--contraction $\Phi : X \ra \R$ such that $\Phi(a)=0$ and $\Phi(b)=r$.
\edf

The assumption that $K$ is nonzero is needed to ensure that $\sigma$ separates points in the setting of coarse median spaces. In the special case where $X$ is a CAT(0) cube complex, we may take $K=0$. More precisely, if $X$ is a CAT(0) cube complex endowed with the piecewise $\ell^p$ length metric for $p \in \{1,2,\infty\}$ for instance, then the new metric $\sigma$ for $K=0$ is the piecewise $\ell^\infty$ length metric on $X$.

\blem
The function $\sigma$ is a metric on $X$.
\elem

\begin{proof}
Let $a,b \in X$ be distinct. Consider the map $\Phi : X \ra \{0,K\}$ that sends $b$ to $K$ and everything else to $0$. It is a $K$--contraction, and so $\sigma(a,b) \geq K > 0$.

The proof of the triangle inequality is identical to~\cite[Lemma~3.1]{bowditch_median_injective}. For the reader's convenience, we repeat it here. Let $a,b,c \in X$. For each $r<\sigma(a,b)$ there exists a $K$--contraction $\Phi_r:X \ra \R$ such that $|\Phi_r(a)-\Phi_r(b)|\ge r$. We certainly have $\sigma(a,c)\ge|\Phi_r(c)-\Phi_r(a)|$ and $\sigma(b,c)\ge|\Phi_r(b)-\Phi_r(c)|$, so
\begin{align*}
\sigma(a,c) +\sigma(c,b) &\geq \sup\{|\Phi_r(c)-\Phi_r(a)|+|\Phi_r(b)-\Phi_r(c)| \st r<\sigma(a,b)\} \\
    &\ge \sup\{|\Phi_r(b)-\Phi_r(a)| \st r<\sigma(a,b)\} \hspace{2mm}=\hspace{2mm} \sigma(a,b). \qedhere
\end{align*}
\end{proof}

\begin{rmk}
Although the construction of $\sigma$ depends on the choice of a positive constant $K$, the actual choice of $K$ will not matter to us here. If $K_1<K_2$, then any $K_1$--contraction is automatically a $K_2$--contraction, so $\sigma_{K_1}\le\sigma_{K_2}$. On the other hand, if $\Phi$ is a $K_2$--contraction, then $\frac{K_1}{K_2}\Phi$ is a $K_1$--contraction, so $\sigma_{K_1}\ge\frac{K_1}{K_2}\sigma_{K_2}$. Thus any two choices of $K$ give biLipschitz metrics.
\end{rmk}

We record the following simple consequence of the definition of $\sigma$.

\blem \label{lem:acting_on_new_metric}
If a group $G$ is acting isometrically on a coarse median space $(X,\mu,d)$ by median isometries, in the sense that $g\mu(x,y,z)=\mu(gx,gy,gz)$ for all $g\in G$, $x,y,z\in X$, then the induced action of $G$ on $(X,\mu,\sigma)$ is isometric.
\elem

\bp
For any $g\in G$ and $x,y\in X$, if $\Phi$ is a $K$--contraction with $\Phi(x)=0$ and $\Phi(y)=r$, then $\Phi'=\Phi g^{-1}$ is a $K$--contraction with $\Phi'(gx)=0$ and $\Phi'(gy)=r$.
\ep

In order to help understand the metric $\sigma$, we shall work with coarse median spaces that have the following property, which is a strengthening of the second axiom of coarse median spaces for sets $A=\{a,b\}$ with cardinality $2$. We require an approximation of the entire median interval $[a,b]$ with uniform constants, and also that the comparison map is a quasi-isometry and not just coarsely invertible.

\bdf[Quasicubical intervals]
Let $(X, \mu, d)$ be a coarse median space. We say that it has \emph{quasicubical intervals} if it has finite rank $\nu$ and there exists $\kappa \geq 1$ such that the following hold. For every $a,b \in X$, there exists a finite CAT(0) cube complex $Q$ of dimension at most $\nu$, endowed with the $\ell^1$ metric $d_Q$ and the median $\mu_Q$, such that there exists a map $\lambda :  Q \ra [a,b]$ satisfying:
\bit
\item $\lambda$ is a $(\kappa,\kappa)$--quasi-isometry, i.e. $\lambda$ is $\kappa$--coarsely onto and
\[\forall \alpha,\beta \in Q, \f{1}{\kappa }d_Q(\alpha,\beta)-\kappa  \leq d(\lambda(\alpha),\lambda(\beta)) \leq \kappa d_Q(\alpha,\beta)+\kappa ;\]
\item $\lambda$ is a $\kappa$--quasi-median homomorphism, i.e.
\[\forall \alpha,\beta,\gamma \in Q, d(\lambda(\mu_Q(\alpha,\beta,\gamma)),\mu(\lambda(\alpha),\lambda(\beta),\lambda(\gamma)) \leq \kappa .\]\eit
\edf

Obviously this is satisfied by finite dimensional CAT(0) cube complexes, or indeed by any space with a global quasi-median quasi-isometry to a CAT(0) cube complex. 

\bpro \label{prop:applies_to_hhss}
Hierarchically hyperbolic spaces have quasicubical intervals, as do coarse median spaces satisfying the axioms (B1)-(B10) in~\cite{bowditch_convexhulls}.
\epro

\bp
In hierarchically hyperbolic spaces, the notion of median intervals used here coincides coarsely with the hierarchically quasiconvex hull of a pair of points defined in \cite{behrstockhagensisto:hierarchically:2}, by~\cite[Corollary~5.12]{hhs_convexity} and \cite[Lemma~8.1]{bowditch_quasiflats}. The first statement is thus a special case of \cite[Theorem~2.1]{hhs_quasiflats}. The second statement is exactly \cite[Theorem~1.3]{bowditch_convexhulls}.
\ep

As noted by Bowditch, every hierarchically hyperbolic space satisfies the axioms (B1)-(B10) in~\cite{bowditch_convexhulls}. It is not known whether all cocompact cube complexes can be given a structure that satisfies these axioms.

\mk

We can now state the main result of this section. It sums up Lemma~\ref{lem:acting_on_new_metric}, Proposition~\ref{prop:bilipschitzmetrics}, Proposition~\ref{prop:coarsely_geodesic}, and Lemma~\ref{lem:balls_median_convex}, and the proof is split over the next three subsections. 

\bthm \label{thm:properties_of_sigma}
Assume that the coarse median space $(X, \mu, d)$ has quasicubical intervals and is roughly geodesic. The metrics $\sigma$ and $d$ are quasi-isometric, $\sigma$ is weakly roughly geodesic, and balls for $\sigma$ are uniformly coarsely median-convex. Moreover, $\sigma$ is invariant under the group of median isometries of $(X,\mu,d)$.
\ethm

\subsection{The metrics \texorpdfstring{$d$}{d} and
  \texorpdfstring{$\sigma$}{σ} are quasi-isometric} ~

Here we shall prove that the new distance $\sigma$ is quasi-isometric to the original distance $d$. We need the following technical result for coarse median spaces, which is a special case of Lemmas~2.18 and~2.19 of \cite{niblowrightzhang:four}.

\blem \label{lem:median_5_points}
In any coarse median space $(X,d,\mu)$, there exists a constant $H_5 \geq 0$ such that the following inequalities hold for any $a,b,x,y,z\in X$.
\beq& & d(\mu(a,b,\mu_{x,y,z}),\mu(\mu_{a,b,x},\mu_{a,b,y},z)) \leq H_5 \\
& & d(\mu(a,b,\mu_{x,y,z}),\mu(\mu_{a,b,x},\mu_{a,b,y},\mu_{a,b,z})) \leq H_5. \eeq
\elem

We will now prove that, up to multiplicative and additive constants, one can restrict to contractions defined on the interval between
two points for the definition of $\sigma$.

\blem \label{lem:sigma'}
For each $a,b \in X$, let $\sigma'(a,b)$ denote the supremum of all $r \geq 0$ such that there exists a $K$--contraction $\Phi' : [a,b] \ra \R$ for which $\Phi'(a)=0$ and $\Phi'(b)=r$. There exists $L \geq 1$ such that, for each $a,b \in X$, we have $\sigma(a,b) \leq \sigma'(a,b) \leq L \sigma(a,b)$.
\elem

\bp
It is immediate that $\sigma(a,b) \leq \sigma'(a,b)$. Consider $r \geq 0$ and a $K$--contraction $\Phi' : [a,b] \ra \R$ such that $\Phi'(a)=0$ and $\Phi'(b)=r$. Define $\Phi : X \ra \R$ by $c \mapsto \Phi'(\mu(a,b,c))$. Since the map $c \mapsto \mu(a,b,c)$ is $(h(0),h(0))$--coarsely Lipschitz and $\Phi'$ is $(1,K)$--coarsely Lipschitz, we deduce that $\Phi$ is $(h(0),h(0)+K)$--coarsely Lipschitz.

\mk

Now let $x,y,z \in X$. According to Lemma~\ref{lem:median_5_points} we have
\[d(\mu(a,b,\mu_{x,y,z}),\mu(\mu_{a,b,x},\mu_{a,b,y},\mu_{a,b,z})) \leq H_5.\]
Hence, since $\Phi'$ is $(1,K)$--coarsely Lipschitz,
\[|\Phi'(\mu(a,b,\mu_{x,y,z})) - \Phi'(\mu(\mu_{a,b,x},\mu_{a,b,y},\mu_{a,b,z}))| \leq H_5 + K.\]
But $\Phi'$ is also a $K$--quasi-median homomorphism, and so
\[|\Phi'(\mu(\mu_{a,b,x},\mu_{a,b,y},\mu_{a,b,z})) - \mu_\R(\Phi'(\mu_{a,b,x}),\Phi'(\mu_{a,b,y}),\Phi'(\mu_{a,b,z}))| \leq K.\]
Combining these and recalling the definition of $\Phi$ enables us to conclude that $|\Phi(\mu_{x,y,z})-\mu_\R(\Phi(x),\Phi(y),\Phi(z))| \leq H_5 +2K$. Thus, if we set $L = \max\{h(0),1+\f{h(0)}{K},2+\f{H_5}{K}\}$, then we have that $\f{1}{L}\Phi$ is a $K$--contraction, and so $\sigma'(a,b) \leq L\sigma(a,b)$.
\ep

We can now deduce that $\sigma$ is quasi-isometric to $d$ in the setting of Theorem~\ref{thm:properties_of_sigma}.

\bpro \label{prop:bilipschitzmetrics}
If $(X,\mu,d)$ has quasicubical intervals, then $d$ and $\sigma$ are quasi-isometric.
\epro

\bp
Fix $a,b \in X$. First of all, since any $K$--contraction is $(1,K)$--coarsely Lipschitz, we have $\sigma(a,b) \leq d(a,b)+K$.

\mk

According to the quasicubicality of intervals, there exists a finite CAT(0) cube complex $Q$ of dimension at most $\nu$, and a map $\lambda : (Q,d_Q) \ra [a,b]$ that is a $(\kappa ,\kappa )$--quasi-isometry and a $\kappa $--quasi-median homomorphism. Then $\lambda$ has a quasi-inverse $\pi : [a,b] \ra (Q,d_Q)$ that is a $(\kappa',\kappa')$--quasi-isometry and a $\kappa'$--quasi-median homomorphism, where $\kappa'$ is a constant depending only on $\kappa$ and $h(0)$.

Note that we shall in fact use $Q$ to denote the vertex set, $d_Q$ to denote the combinatorial (piecewise $\ell^1$) distance on $Q$, and $\mu_Q$ to denote the median on $Q$. Let us denote by $\sigma_Q$ the piecewise $\ell^\infty$ distance on $Q$: we have $\sigma_Q \leq d_Q \leq \nu \sigma_Q$.

\mk

Since $Q$ is a CAT(0) cube complex, there exists a $0$--contraction $\Phi_Q : (Q,d_Q) \ra \Z$ such that $\Phi_Q(\pi(a))=0$ and $\Phi_Q(\pi(b)) = \sigma_Q(\pi(a),\pi(b))$ (see~\cite[\S7]{bowditch_median_injective}, \cite[Cor.~2.5]{bandeltvandevel:superextensions}). Let us consider $\Phi'=\f{\min\{1,K\}}{\kappa'}\Phi_Q \pi : [a,b] \ra \R$. Since $\pi$ is a $(\kappa',\kappa')$--quasi-isometry and $\Phi_Q$ is $1$--Lipschitz, we deduce that $\Phi'$ is $(1,K)$--coarsely Lipschitz. Furthermore, for every $x,y,z \in [a,b]$, we have:
\beq & &|\Phi_Q \pi(\mu_{x,y,z}) - \mu_\R(\Phi_Q\pi(x),\Phi_Q\pi(y),\Phi_Q\pi(z))|\\
 & \leq &|\Phi_Q\pi(\mu_{x,y,z})-\Phi_Q(\mu_Q(\pi(x),\pi(y),\pi(z)))| \\
 &  &+\quad |\Phi_Q (\mu_Q(\pi(x),\pi(y),\pi(z))) - \mu_\R(\Phi_Q\pi(x),\Phi_Q\pi(y),\Phi_Q\pi(z))| \\
 & \leq &d_Q(\pi(\mu_{x,y,z}),\mu_Q(\pi(x),\pi(y),\pi(z))) \quad\le\quad \kappa',\eeq
 so $\Phi'$ is $K$--quasi-median.
 
The map $\Phi'$ is therefore a $K$--contraction on $[a,b]$, and $\Phi'(a)=0$ and $\Phi'(b)=\f{\min\{1,K\}}{\kappa'}\sigma_Q(\pi(a),\pi(b)) \geq \f{\min\{1,K\}}{\nu \kappa'}d_Q(\pi(a),\pi(b))$. Using Lemma~\ref{lem:sigma'}, we deduce that $d_Q(\pi(a),\pi(b))\le \f{\nu \kappa'L}{\min\{1,K\}}\sigma(a,b)$. But $\pi$ is a $(\kappa',\kappa')$--quasi-isometry, so we also have $d_Q(\pi(a),\pi(b)) \ge \f{1}{\kappa'}d(a,b)-\kappa'$.

\mk

In conclusion, we have 
\[
\f{\min\{1,K\}}{\nu {\kappa'}^2L}d(a,b) - \f{\min\{1,K\}}{\nu L} 
	\quad\le\quad \sigma(a,b) \quad\le\quad d(a,b)+K
\]
for all $a,b\in X$.
\ep

\subsection{The metric \texorpdfstring{$\sigma$}{σ} is weakly roughly geodesic} ~

Recall that $(X,\mu,d)$ is a coarse median space with corresponding function $h$, that $X$ has quasicubical intervals (though this will only be used for Proposition~\ref{prop:coarsely_geodesic} in this section), and that the metric $d$ is $C_d$--roughly geodesic. We shall prove that the new metric $\sigma$ is weakly roughly geodesic (see Definition~\ref{def:wrg}). This will be the most difficult part of the proof of Theorem~\ref{thm:properties_of_sigma}.

\mk

Let $a,b \in X$, let $E$ be a small positive constant, and consider $K$--contractions $\Phi_1 : X \ra [0,r]$ and $\Phi_2 : X \ra [r,r+s]$ (for some $r,s \geq E$) such that $\Phi_1(a)\le E$ and $\Phi_2(b)\ge r+s-E$. We want to find a criterion to ensure that we can combine $\Phi_1$ and $\Phi_2$ into a contraction $\Phi$ such that $\Phi(a)=0$ and $(r+s)-\Phi(b)$ is bounded above by some constant.

\blem \label{lem:join_of_contractions}
Assume that $a,b,\Phi_1,\Phi_2,r,s,E$ are as above. Let $D=h(0)(3K+4C_d)+4K+h(0)$. If $t \in [0,\min\{r,s\}-D+K-E]$ is such that the sets
\[Z_1=\{z \in X \st \Phi_1(z) \leq r-t-K\} \quad\mbox{ and }\quad Z_2=\{z \in X \st \Phi_2(z) \geq r+t+K\}\]
are disjoint, then $\sigma(a,b) \geq r+s-2t-2D-2E$.
\elem

\bp
For $m\in\{0,1,2\}$, let us write $Y^m_1 = \{x \in X \st \Phi_1(x) \leq r-t-D+mK\}$ and $Y^m_2=\{x \in X \st \Phi_2(x) \geq r+t+D-mK\}$. Note that if $m_1<m_2$, then $Y_i^{m_1}\subset Y_i^{m_2}$.

\mk

{\bf Claim 1: }$d(Y^2_1,Y^2_2) \geq D-4K$.

\mk

{\bf Proof of Claim 1:}
Let $x_1 \in Y^2_1$ and $x_2 \in Y^2_2$. Since $Y^2_2 \subset Z_2$, we have $x_2 \not\in Z_1$, so $\Phi_1(x_2) > r-t-K$. We also have $\Phi_1(x_1) \leq r-t-D+2K$, so $|\Phi_1(x_1) - \Phi_1(x_2)| \geq D-3K$. As $\Phi_1$ is $(1,K)$--coarsely Lipschitz, we have $|\Phi_1(x_1) - \Phi_1(x_2)| \leq d(x_1,x_2)+K$, and hence $d(x_1,x_2) \geq D-4K. \hfill\diamondsuit$

\mk

{\bf Claim 2: }$d(Y^1_1,Y^1_2) \geq 3K+4C_d$.

\mk

{\bf Proof of Claim 2:}
Let $x_1 \in Y^1_1$ and $x_2 \in Y^1_2$, and set $y_1 = \mu(a,b,x_1) \in [a,b]$ and $y_2 = \mu(a,b,x_2) \in [a,b]$. We know that $\Phi_1(y_1) \leq \mu_\R(\Phi_1(a),\Phi_1(b),\Phi_1(x_1))+K$. We also have $\Phi_1(a)\le E$ by assumption, and $\Phi_1(x_1) \leq r-t-D+K$. As this latter quantity is at least $E$, $\mu_\R(\Phi_1(a),\Phi_1(b),\Phi_1(x_1)) \leq r-t-D+K$. Hence $\Phi_1(y_1) \leq r-t-D+2K$, so $y_1 \in Y^2_1$. A similar argument shows that $y_2 \in Y^2_2$.

According to Claim 1, $d(y_1,y_2) \geq D-4K$. Since $\mu$ is $(h(0),h(0))$--coarsely Lipschitz with respect to each variable, we have $d(y_1,y_2) \leq h(0) d(x_1,x_2)+h(0)$, so $d(x_1,x_2) \geq \f{d(y_1,y_2)-h(0)}{h(0)} \geq \f{D-4K-h(0)}{h(0)}=3K+4C_d$, as desired. $\hfill\diamondsuit$ 

\mk

{\bf Claim 3:} \emph{The set $\{\mu_{x,y,z} \st x,y \in Y^0_1,z\in X\}$ is disjoint from $Y^0_2$, and $\{\mu_{x,y,z} \st x,y \in Y^0_2,z\in X\}$ is disjoint from $Y^0_1$.}

\mk

{\bf Proof of Claim 3:} Fix $x,y \in Y^0_1$ and $z \in X$. Since $\Phi_1(x),\Phi_1(y) \leq r-t-D$, we deduce that $\mu_\R(\Phi_1(x),\Phi_1(y),\Phi_1(z)) \leq r-t-D$, and it follows that $\Phi_1(\mu_{x,y,z}) \leq r-t-D+K$, so $\mu_{x,y,z} \in Y^1_1$. Because we showed in Claim~2 that $d(Y^1_1,Y^1_2) \geq 3K+4C_d >0$, we know that $\mu_{x,y,z} \not\in Y^1_2$, and, in particular, $\mu_{x,y,z} \not\in Y^0_2$. The other case is similar. $\hfill\diamondsuit$

\mk

Write $Y = X \smallsetminus (Y^0_1 \cup Y^0_2)$, and consider $\Phi : X \ra [0,r+s-2t-2D]$ defined by:
\beq
\text{If }x \in Y^0_1 & \text{then} & \Phi(x)=\Phi_1(x). \\
\text{If }x \in Y^0_2 & \text{then} & \Phi(x)=\Phi_2(x)-2t-2D. \\
\text{If }x \in Y     & \text{then} & \Phi(x)=r-t-D.
\eeq
We have $\Phi(a)\le E$ and $\Phi(b)\ge r+s-2t-2D-E$, so if we prove that $\Phi$ is a $K$--contraction, then we may deduce that $\sigma(a,b) \geq r+s-2t-2D-2E$, the desired conclusion.

\mk

{\bf Claim 4:} \emph{$\Phi$ is $(1,K)$--coarsely Lipschitz.}

\mk

{\bf Proof of Claim 4:} Notice that $\Phi$ coincides on $Y^0_1 \cup Y$ with the composition of $\Phi_1 :X \ra [0,r]$ with the 1--Lipschitz map $m_t=\min(\cdot,r-t-D) : [0,r] \ra [0,r-t-D]$. Hence, if $x,y\in Y^0_1\cup Y$, then $|\Phi(x)-\Phi(y)| \leq |\Phi_1(x) - \Phi_1(y)| \leq d(x,y) + K$. A similar argument involving a maximum function applies if $x,y \in Y^0_2 \cup Y$.

Now suppose that $x \in Y^0_1$ and $y \in Y^0_2$. Since $d$ is $C_d$--roughly geodesic, there is a $(1,C_d)$--quasi-isometric embedding $f:[0,d(x,y)]\to X$ with $f(0)=x$, $f(d(x,y))=y$. For any $\eps>0$, there exists $\tau$ such that $f(\tau)\in Y^0_1$ but $f(\tau+\delta)\not\in Y^0_1$ for any $\delta>\eps$. (Were $f$ continuous, we could take $\eps=0$ and use the maximal $\tau$ with $f(\tau)\in Y^0_1$.) Write $z_1=f(\tau)$. We have $d(x,z_1)+d(z_1,y)\le d(x,y)+C_d$ and $\Phi_1(z_1)\le r-t-D$. Moreover, for any $\delta>\eps$ we have
\[\Phi_1(z_1) \hspace{2mm}\ge\hspace{2mm} \Phi_1(f(\tau+\delta))-(d(f(\tau),f(\tau+\delta))+K) \hspace{2mm}>\hspace{2mm} r-t-D-(\delta+C_d+K),\]
and so $\Phi_1(z_1)\ge r-t-D-C_d-K-\eps$. We can now similarly construct $z_2\in Y^0_2$ such that $d(z_1,z_2)+d(z_2,y)\le d(z_1,y)+C_d$ and $r+t+D\le\Phi_2(z_2)\le r+t+D+C_d+K+\eps$.

With these, we can compute
\beq |\Phi(x)-\Phi(y)| 
    &\le& |\Phi(x)-\Phi(z_1)| + |\Phi(z_1)-\Phi(z_2)| + |\Phi(z_2)-\Phi(y)| \\
    &=& |\Phi_1(x)-\Phi_1(z_1)| + |\Phi_1(z_1)-(\Phi_2(z_2)-2t-2D)| + |\Phi_2(z_2)-\Phi_2(y)| \\
    & \leq & (d(x,z_1) + K) + (2C_d+2K+2\eps) + (d(z_2,y) + K) \\
    & \leq & (d(x,y)+C_d-d(z_1,y)+K) + (2C_d+2K+2\eps) + (d(z_1,y)+C_d-d(z_1,z_2)+K) \\
    & = & d(x,y) -d(z_1,z_2) +4K +4C_d +2\eps \\
    & \leq & d(x,y) + K +2\eps, \eeq
where the last line comes from Claim~2: $d(z_1,z_2) \geq d(Y^0_1,Y^0_2) \geq d(Y^1_1,Y^1_2) \ge 3K+4C_d$. This is sufficient, because $\eps$ can be taken to be arbitrarily close to 0. $\hfill\diamondsuit$

\mk

{\bf Claim 5:} \emph{$\Phi$ is $K$--quasi-median.}

\mk

{\bf Proof of Claim 5:} As noted in the proof of Claim 4, on $Y^0_1\cup Y$ we have $\Phi=m_t\Phi_1$. As $m_t$ is a median homomorphism with respect to $\mu_\R$, if $x,y,z\in Y^0_1\cup Y$, then
\begin{align*} 
|\Phi(\mu_{x,y,z}) - \mu_\R(\Phi(x),\Phi(y),\Phi(z))| 
&= |m_t\Phi_1(\mu_{x,y,z}) - \mu_\R(m_t\Phi_1(x),m_t\Phi_1(y),m_t\Phi_1(z))| \\
&\leq |\Phi_1(\mu_{x,y,z}) - \mu_\R(\Phi_1(x),\Phi_1(y),\Phi_1(z))| \hspace{2mm}\le\hspace{2mm} K,
\end{align*}
and similarly if $x,y,z\in Y^0_2\cup Y$.

Assume now that $x,y \in Y^0_1$ and $z \in Y^0_2$. We ave that both $\Phi(x)$ and $\Phi(y)$ are at most $r-t-D$. Moreover, $\Phi(z)=\Phi_2(z)-2t-2D \geq r-t-D$, and the fact that $z\not\in Y^0_1$ implies that $\Phi_1(z)>r-t-D$. Thus $\mu_\R(\Phi(x),\Phi(y),\Phi(z))=\mu_\R(\Phi_1(x),\Phi_1(y),\Phi_1(z))\le r-t-D$. As $\Phi_1$ is $K$--quasi-median, we deduce that $|\mu_\R(\Phi(x),\Phi(y),\Phi(z)) - \Phi_1(\mu_{x,y,z})| \leq K$. According to Claim~3, we know that $\mu_{x,y,z} \not\in Y^0_2$, and so $\Phi(\mu_{x,y,z}) = m_{t}\Phi_1(\mu_{x,y,z})$. But $\mu_\R(\Phi(x),\Phi(y),\Phi(z)) \leq r-t-D$, so we conclude that $|\mu_\R(\Phi(x),\Phi(y),\Phi(z)) - \Phi(\mu_{x,y,z})| \leq K$. A similar argument applies when $x,y\in Y^0_2$ and $z\in Y^0_1$.

Assume finally that $x \in Y^0_1$, $y \in Y$, and $z \in Y^0_2$. Since $\Phi(x)=\Phi_1(x) \leq r-t-D$, $\Phi(y)=r-t-D$ and $\Phi(z)=\Phi_2(z)-2t-2D \geq r-t-D$, we have $\mu_\R(\Phi(x),\Phi(y),\Phi(z)) = r-t-D$. If $\mu_{x,y,z} \in Y$, then $\Phi(\mu_{x,y,z})=r-t-D=\mu_\R(\Phi(x),\Phi(y),\Phi(z))$. If $\mu_{x,y,z} \in Y^0_1$, then $\Phi(\mu_{x,y,z})=\Phi_1(\mu_{x,y,z}) \geq \mu_\R(\Phi_1(x),\Phi_1(y),\Phi_1(z))-K \geq r-t-D-K$, from which it follows that $|\mu_\R(\Phi(x),\Phi(y),\Phi(z)) - \Phi(\mu_{x,y,z})| \leq K$. A similar argument applies if $\mu_{x,y,z} \in Y^0_2$. 

We have shown that $|\mu_\R(\Phi(x),\Phi(y),\Phi(z)) - \Phi(\mu_{x,y,z})| \leq K$ in all cases. $\hfill\diamondsuit$

\mk

We have proved that $\Phi$ is a $K$--contraction. As stated above, this shows that $\sigma(a,b) \geq |\Phi(a)-\Phi(b)|\ge r+s-2t-2D-2E$.
\ep

Recall that the \emph{convex hull} $\Hull(A)$ of a subset $A$ of a CAT(0) cube complex $Q$ is the smallest convex subcomplex of $Q$ containing $A$. Equivalently, it is the smallest subcomplex that contains $A$ and is \emph{median convex}, in the sense that $\mu(q,a,b)\in\Hull(A)$ whenever $a,b\in\Hull(A)$. We regard a subset of $Q^{(0)}$ as convex if the full subcomplex spanned by it is convex. We need the following iterative description convex hulls in CAT(0) cube complexes. 

\blem \label{lem:iterated_median}
Let $Q$ be a CAT(0) cube complex of dimension at most $\nu$, and let $\mu_Q : {Q^{(0)}}^3 \ra Q^{(0)}$ denote the median. Given $A \subset Q^{(0)}$, set $A_0=A$, and for each $i \in \N$, let
\[A_{i+1} = \mu_Q(Q^{(0)},A_i,A_i) = \{\mu_Q(x,a,b) \st a \in A_i, b \in A_i, x \in Q^{(0)}\}.\]
Then $A_{\nu'} = \Hull(A)$, where $\nu'=\max\{1,\nu-1\}$.
\elem

Note that the constant $\nu'$ is probably far from optimal. However, $\nu'$ does depend on $A$ and $Q$. For example, if $A$ is the star of a vertex in a $\nu$--cube, then it can be seen that the optimal value of $\nu'$ is $\lceil\log_2\nu\rceil$ in this case. 

\bp

The result is trivial if $A$ is convex. Otherwise, fix $x\in\Hull(A)\smallsetminus A$, and let $\mathcal H$ be the collection of hyperplanes of $\Hull(A)$ that are adjacent to $x$. For each $H\in\mathcal H$, let $Q^{(0)}=H^+\sqcup H^-$ denote the partition defined by $H$, where $x\in H^+$. Let $\{H_1,\dots,H_n\}$ be a maximal pairwise crossing family in $\mathcal H$. We have $n\le\nu$. For each $i$, let $\mathcal H_i$ denote the set of elements of $\mathcal H$ that are disjoint from $H_i$, together with $H_i$. An important observation is that $H_i^-\subset H^+$ whenever $H\in\mathcal H_i\smallsetminus\{H_i\}$. 

If $n=1$, then $x$ is a cut-point or leaf of $\Hull(A)$, so taking any $a\in A\cap H^+$, $b\in A\cap H^-$ gives $x=\mu(x,a,b)$, and we are done. 

So suppose that $n\ge2$. If for every $a\in A\cap H_1^-$ we have $a\in H_2^-$, then for every $b\in A\cap H_2^+$ we have $b\in H_1^+$, so if we take $z_1\in A\cap H_1^-$ and $z_2\in A\cap H_2^+$, then $\mu(x,z_1,z_2)\in H^+\cap H'^+$ for every $H\in\mathcal H_1$, $H'\in\mathcal H_2$. We can reason similarly if every element of $A\cap H_2^-$ lies in $H_1^-$. Otherwise there exist $z_1\in A\cap H_1^-\cap H_2^+$ and $z_2\in A\cap H_1^+\cap H_2^-$, and we again have $\mu(x,z_1,z_2)\in H^+\cap H'^+$ for every $H\in\mathcal H_1$, $H'\in\mathcal H_2$. Let $y_1=\mu(x,z_1,z_2)\in A_1$.

We proceed inductively. Suppose that we have $y_i\in A_i$ such that $y_i\in H^+$ for all $H\in\bigcup_{j\le i+1} \mathcal H_j$. Let $z_{i+2}$ be any point of $A$ that is separated from $y_i$ by $H_{i+2}$, and set $y_{i+1}=\mu(x,y_i,z_{i+2})$. Since $x,y_i\in H^+$ for every $H\in\bigcup_{j\le i+1} \mathcal H_j$, the same is true of $y_{i+1}$, and since $y_i$ and $z_{i+2}$ lie on opposite sides of $H_{i+2}$, we also have that $y_{i+1}\in H^+$ for all $H\in\mathcal H_{i+2}$.

By this procedure, we obtain $y_{n-1}\in A_{n-1}\cap\Hull(A)$ that is not separated from $x$ by any hyperplane of $\Hull(A)$, so we must have $y_{n-1}=x$.
\ep

In order to apply Lemma~\ref{lem:join_of_contractions}, we focus on contractions on CAT(0) cube complexes. Recall that a \emph{chain} of hyperplanes is an ordered sequence $(H_1,\dots,H_n)$ of pairwise disjoint hyperplanes such that $H_j$ separates $H_i$ from $H_k$ whenever $i<j<k$.

\blem \label{lem:approximation_contraction}
Let $Q$ be a CAT(0) cube complex of dimension at most $\nu$, and let $\Phi : Q^{(0)} \ra \R$ be a $K'$--quasi-median, $(K',K')$--coarsely Lipschitz map (for the $\ell^1$ metric) with bounded image. There exists an interval $[u,v]$ of $\Z$ and a chain $(H_n)_{u \leq n \leq v}$ of hyperplanes of $Q$ satisfying the following:
\bit
\item for each vertex $x$ in $Q$, there exists a unique $n=\Psi(x) \in [u-1,v]$ such that:
\bit
\item either $u \leq n \leq v-1$ and $x$ is between $H_n$ and $H_{n+1}$,
\item or $n=u-1$ and $H_u$ separates $x$ from $H_{u+1}$,
\item or $n=v$ and $H_v$ separates $x$ from $H_{v-1}$, and
\eit
\item for each vertex $x$ in $Q$, we have $|\Phi(x) - 4K'\nu\Psi(x)| \leq 4K'\nu$.
\eit
\elem

\bp
Fix $n \in \Z$, and consider $K_n=\Phi^{-1}((2An-A,2An])$, where $A=2K'\nu$. Since $Q^{(0)}$ is $1$--connected, we know that $\Phi(Q^{(0)})$ is $2K'$--connected. In particular, the set of integers $n \in \Z$ such that $K_n \neq \emptyset$ is an interval $[u-1,v]$. Furthermore, for each $u \leq n \leq v-1$, we know that $K_n$ disconnects $Q$.

In the notation of Lemma~\ref{lem:iterated_median} we have that, for all $i\ge0$, if $x \in (K_n)_i$, then $|\Phi(x)-\Phi(K_n)| \leq K'i$. Indeed, this is clear for $i=0$, and if $x\in(K_n)_{i+1}$, so that there exist $a,b\in(K_n)_i$ with $x=\mu_Q(x,a,b)$, then the fact that $\Phi$ is $K'$--quasimedian implies that $|\Phi(x) - \mu_\R(\Phi(x),\Phi(a),\Phi(b))|\le K'$, yielding the claimed inequality by induction. In particular, Lemma~\ref{lem:iterated_median} tells us that every $x\in\Hull(K_n)$ satisfies $|\Phi(x)-\Phi(K_n)|\le K'\nu$. 

As a consequence, if $n \neq m$, then the convex subcomplexes $\Hull(K_n)$ and $\Hull(K_m)$ are disjoint. Thus, for each $u \leq n \leq v$, there exists a hyperplane $H_n$ of $Q$ that separates $\Hull(K_{n-1})$ from $\Hull(K_n)$ \cite[Corollary~1]{chepoi:convexity}. 

For each vertex $x \in Q^{(0)}$, let $u-1 \leq n \leq v$ be such that $\Phi(x) \in (2An-2A,2An]$. Then $\Psi(x)$ is either equal to $n-1$ or to $n$. So $|\Phi(x) - 2A\Psi(x)| \leq 2A=4K'\nu$.
\ep

Before stating the next lemma, we remark that, given any chain $\mathcal H$ of hyperplanes in a finite CAT(0) cube complex $Q$, there is an associated map $Q^{(0)}\to\Z$: the cube complex dual to $\mathcal H$ is a finite interval of $\Z$, and each vertex of $Q$ determines a consistent orientation of the hyperplanes in $\mathcal H$. This is a special case of the \emph{restriction quotient} described in \cite{capracesageev:rank}, and it is clearly a median map. Conversely, any $0$--contraction on $Q$ can be realised as restriction quotient in this manner. Moreover, after a translation of $\Z$, we may assume that the codomain is contained in $\N$ if it is bounded.

\blem \label{lem:midpoint_in_CCC}
Let $Q$ be a finite CAT(0) cube complex of dimension at most $\nu$. Let ${\mathcal C}$ be a (necessarily finite up to translations of $\Z$) family of $0$--contractions on $Q$, i.e. each $\Psi \in {\mathcal C}$ is a map $Q^{(0)} \ra \N$ given by a chain $(H_{\Psi,1},\dots,H_{\Psi,n_\Psi})$ of hyperplanes of $Q$. Let $\sigma_{\mathcal C}$ denote the pseudometric on $Q^{(0)}$ defined by
\[\forall \alpha,\beta \in Q^{(0)}, \sigma_{\mathcal C}(\alpha,\beta) = \max_{\Psi \in {\mathcal C}} |\Psi(\alpha)-\Psi(\beta)|.\]
Then for each $\alpha,\beta \in Q^{(0)}$ and for each integer $0 \leq r \leq \sigma_{\mathcal C}(\alpha,\beta)$, there is a vertex $\gamma \in [\alpha,\beta]$ and contractions $\Psi_1,\Psi_2 \in {\mathcal C}$ such that the following hold.
\ben
\item $\sigma_{\mathcal C}(\alpha,\gamma)=r$,
\item $\sigma_{\mathcal C}(\alpha,\gamma)=|\Psi_1(\alpha)-\Psi_1(\gamma)|$ and $\sigma_{\mathcal C}(\gamma,\beta)=|\Psi_2(\gamma)-\Psi_2(\beta)|$,
\item if $(H_{1,1},\dots,H_{1,n_1})$ is the maximal subchain of hyperplanes defining $\Psi_1$ that separate $\alpha$ from $\gamma$, and $(H_{2,1},\dots,H_{2,n_2})$ is the maximal subchain of hyperplanes defining $\Psi_2$ that separate $\gamma$ from $\beta$, then $(H_{1,1},\dots,H_{1,n_1},H_{2,1},\dots,H_{2,n_2})$ is a chain. \label{item:disjoint_hyperplanes}
\een
\elem

\bp
Fix $\alpha,\beta \in Q^{(0)}$ and an integer $0 < r < \sigma_{\mathcal C}(\alpha,\beta)$. Since $\sigma_{\mathcal C}$ is $1$--Lipschitz with respect to the combinatorial distance on $Q^{(0)}$, we know that there exists $\gamma \in [\alpha,\beta]$ such that $\sigma_{\mathcal C}(\alpha,\gamma)=r$. Among all possible choices, choose such $\gamma$ as far away from $\alpha$ as possible, in the sense that:
\[ \mbox{if } \gamma' \in [\alpha,\beta] \mbox{ has } \sigma_{\mathcal C}(\alpha,\gamma')=r \mbox{ and } \gamma \in [\alpha,\gamma'], \mbox{ then } \gamma'=\gamma.\]

Let $\Psi_2 \in {\mathcal C}$ such that $\sigma_{\mathcal C}(\gamma,\beta)=|\Psi_2(\gamma)-\Psi_2(\beta)|$. Let $(H_{2,1},\dots,H_{2,n_2})$ be the maximal subchain of hyperplanes defining $\Psi_2$ that separate $\gamma$ from $\beta$, numbered from $\gamma$ to $\beta$.

Let $H$ be a hyperplane of $Q$ adjacent to $\gamma$ and either equal to $H_{2,1}$ or separating $\gamma$ from $H_{2,1}$, and let $\gamma' \in [\alpha,\beta]$ be the vertex adjacent to $\gamma$ such that $H$ crosses the edge $[\gamma,\gamma']$. First note that, since $H_{2,1}$ separates $\gamma$ and $\beta$, we deduce that $H$ separates $\gamma$ and $\beta$. Thus $H$ does not separate $\alpha$ and $\gamma$, because $\gamma \in [\alpha,\beta]$. In particular, $\gamma \in [\alpha,\gamma']$. Since $\gamma$ is chosen as far from $\alpha$ as possible among points at $\sigma_{\mathcal C}$--distance equal to $r$, and every hyperplane separating $\alpha$ and $\gamma$ separates $\alpha$ and $\gamma'$, we deduce that $\sigma_{\mathcal C}(\alpha,\gamma') > \sigma_{\mathcal C}(\alpha,\gamma)=r$, so $\sigma_{\mathcal C}(\alpha,\gamma')=\sigma_{\mathcal C}(\alpha,\gamma)+1$.

Let $\Psi_1 \in {\mathcal C}$ such that $\sigma_{\mathcal C}(\alpha,\gamma')=|\Psi_1(\alpha)-\Psi_1(\gamma')|$. Let $(H_{1,1},\dots,H_{1,n_1+1})$ be the maximal subchain of hyperplanes defining $\Psi_1$ that separate $\alpha$ from $\gamma'$, numbered from $\alpha$ to $\gamma'$. Since $\sigma_{\mathcal C}(\alpha,\gamma')=\sigma_{\mathcal C}(\alpha,\gamma)+1$, we know that $H=H_{1,n_1+1}$ and that $\sigma_{\mathcal C}(\alpha,\gamma)=|\Psi_1(\alpha)-\Psi_1(\gamma)|$. In particular, $H$ is disjoint from $H_{1,1},\dots,H_{1,n_1}$. We deduce that $H$ separates $H_{1,1},\dots,H_{1,n_1}$ from $H_{2,1},\dots,H_{2,n_2}$, and the conclusion follows. 
\ep

We can now use these lemmas to prove that, in the setting of Theorem~\ref{thm:properties_of_sigma}, the metric $\sigma$ is weakly roughly geodesic (Definition~\ref{def:wrg}).

\bpro\label{prop:coarsely_geodesic} If $(X,\mu,d)$ has quasicubical intervals and is roughly geodesic, then $\sigma$ is weakly roughly geodesic.
\epro

\bp
Let $a,b\in X$. Since $X$ has quasicubical intervals, there exists a finite CAT(0) cube complex $Q$ (with the $\ell^1$ metric) of dimension at most $\nu$, and a map $\lambda : Q \ra [a,b]$ that is a $(\kappa ,\kappa)$--quasi-isometry and a $\kappa $--quasi-median homomorphism. We can therefore fix $\alpha,\beta \in Q$ such that $d(\lambda(\alpha),a) \leq \kappa $ and $d(\lambda(\beta),b) \leq \kappa $. According to Proposition~\ref{prop:bilipschitzmetrics}, there is a constant $q\ge1$ such that $d$ and $\sigma$ are $(q,q)$--quasi-isometric. It follows that $\sigma(\lambda(\alpha),a)\le q(\kappa+1)$ and $\sigma(\lambda(\beta),b)\le q(\kappa+1)$.

For each $K$--contraction $\Phi : X \ra \R$, the composition $\Phi\lambda : Q \ra \R$ is a $K'$--quasi-median, $(K',K')$--coarsely Lipschitz map, where $K'=K+\kappa $. According to Lemma~\ref{lem:approximation_contraction}, there exists a $0$--contraction $\Psi : Q \ra \Z$ such that $|\Phi\lambda(\xi) - 4K'\nu \Psi(\xi)| \leq 4K'\nu$ for all $\xi \in Q^{(0)}$. Let ${\mathcal C}$ denote the set of all $0$--contractions $\Psi : Q \ra \Z$ such that there is some $K$--contraction $\Phi : X \ra \Z$ with $|\Phi\lambda(\xi) - 4K'\nu \Psi(\xi)| \leq 4K'\nu$ for all $\xi \in Q^{(0)}$.

\mk

We shall prove that $\sigma$ is weakly roughly geodesic with constant
\[C'_\sigma =64K'\nu + 4q(\kappa +1) + 4\kappa + 4K + 2D,\]
where $D$ is the constant from Lemma~\ref{lem:join_of_contractions}.

\mk

Let $r\in[0,\sigma(a,b)]$. If $r<C'_\sigma$, then clearly we can take $c=a$ for the desired point. Similarly, if $r>\sigma(a,b)-C'_\sigma$, then we can take $c=b$. Otherwise, Lemma~\ref{lem:midpoint_in_CCC} applied to $\alpha$, $\beta$, the family ${\mathcal C}$, and $r'=\lfloor \frac{r}{4K'\nu} \rfloor$ provides a vertex $\gamma \in [\alpha,\beta]$ and $0$--contractions $\Psi_1,\Psi_2 \in {\mathcal C}$. Let $c=\lambda(\gamma) \in [a,b]$.

\mk

Let us start by computing $\sigma(a,c)$. By definition of the set $\mathcal C$, for any $K$--contraction $\Phi:X\to\R$ there is some $\Psi\in\mathcal C$ (and vice versa: for any $\Psi\in\mathcal C$ there exists a $K$--contraction $\Phi$) such that 
\begin{align*}
\big||\Phi\lambda(\xi)-\Phi\lambda(\zeta)| &- 4K'\nu|\Psi(\xi)-\Psi(\zeta)|\big| \\
\le\hspace{2mm} & \big||\Phi\lambda(\xi)-4K'\nu\Psi(\xi)| + |\Phi\lambda(\zeta)-4K'\nu\Psi(\zeta)|\big| 
        \hspace{2mm}\le\hspace{2mm} 8K'\nu 
\end{align*}
holds for all $\xi,\zeta\in Q^{(0)}$. It follows that 
\begin{align}
|\sigma(\lambda(\xi),\lambda(\zeta)) - 4K'\nu\sigma_{\mathcal C}(\xi,\zeta)| \leq 8K'\nu. \label{eqn:sigma_is_sigma_C}
\end{align}
By the choice of $\gamma$, we have $\sigma_{\mathcal C}(\alpha,\gamma)=r'$. Thus, from \eqref{eqn:sigma_is_sigma_C} we obtain
\begin{align*}
|\sigma(a,c)-r| &\le |\sigma(\lambda(\alpha),\lambda(\gamma))-4K'\nu r'| + q(\kappa+1) + 4K'\nu \nonumber \\
    &\le 12K'\nu +q(\kappa+1) \hspace{2mm}\le\hspace{2mm} C'_\sigma. 
\end{align*}

The aim for the rest of the proof is to confirm the second restriction on $c$, that $\sigma(a,c)+\sigma(c,b)\le\sigma(a,b)+C'_\sigma$. The strategy is to apply Lemma~\ref{lem:join_of_contractions}. 

\mk

Recall that $\Psi_1,\Psi_2 \in {\mathcal C}$ are the $0$--contractions provided by Lemma~\ref{lem:midpoint_in_CCC}: they satisfy $\sigma_{\mathcal C}(\alpha,\gamma)=|\Psi_1(\alpha)-\Psi_1(\gamma)|=r'$ and $\sigma_{\mathcal C}(\gamma,\beta)=|\Psi_2(\gamma)-\Psi_2(\beta)|=s'$. After translations of $\Z$, we may also assume that $\Psi_1(\alpha)=0$, $\Psi_1(\gamma)=\Psi_2(\gamma)=r'$, and $\Psi_2(\beta)=r'+s'$. By definition of ${\mathcal C}$, there exist $K$--contractions $\Phi_1$ and $\Phi_2$ on $X$ such that $|\Phi_1\lambda(\xi) - 4K'\nu \Psi_1(\xi)| \leq 4K'\nu$ and $|\Phi_2\lambda(\xi) - 4K'\nu \Psi_2(\xi)| \leq 4K'\nu$ for all $\xi \in Q^{(0)}$. In particular, $\Phi_1(a) \le \Phi_1\lambda(\alpha)+\kappa+K \le 4K'\nu+\kappa+K$. Moreover, by using \eqref{eqn:sigma_is_sigma_C} we see that 
\begin{align*}
\Phi_2(b)&\ge 4K'\nu(r'+s')-4K'\nu-\kappa-K \\
    &= 4K'\nu(\sigma_{\mathcal C}(\alpha,\gamma)+\sigma_{\mathcal C}(\gamma,\beta))-4K'\nu-\kappa-K \\
    &\ge \sigma(\lambda(\alpha),\lambda(\gamma))+\sigma(\lambda(\gamma),\lambda(\beta)) 
        -16K'\nu - 4K'\nu -\kappa-K \\
    &\ge \sigma(a,c)+\sigma(c,b) - 20K'\nu -2q(\kappa+1)-\kappa-K.
\end{align*}

We are now in the setting of Lemma~\ref{lem:join_of_contractions}, with $E=20K'\nu+2q(\kappa+1)+\kappa+K$ and with the image of $\Phi_2$ being bounded above by $\sigma(a,c)+\sigma(c,b)$. Let us show that the assumptions of the lemma are met if we take $t= 12K'\nu+\kappa+K$. 

\mk

We must first note that $r-D+K-E\ge C'_\sigma-D+K-E\ge t$, and secondly that $\sigma(a,c)+\sigma(c,b)-r-D+K-E \ge\sigma(a,b)-r-D+K-E \ge C'_\sigma-D+K-E \ge t$.

It remains to prove that the subspaces $Z_1 = \{z\in X \st \Phi_1(z) \leq r-t-K\}$ and $Z_2=\{z\in X \st \Phi_2(z) \geq r+t+K\}$ are disjoint.
Fix $z \in X$, let $x=\mu(z,a,b)$, and pick any $\xi \in Q^{(0)}$ such that $d(\lambda(\xi),x) \leq \kappa $.

If $z\in Z_1$, so that $\Phi_1(x) \leq \Phi_1(z)+K\le r-t$, then $\Phi_1(\lambda(\xi)) \leq r-t+\kappa +K$, and hence $\Psi_1(\xi) \leq \f{r-t+\kappa +K+4K'\nu}{4K'\nu} \leq \f{r-8K'\nu}{4K'\nu} \leq r'-1$. Similarly, if $z\in Z_2$, then $\Psi_2(\xi) \geq r'+1$. 

According to Property~\eqref{item:disjoint_hyperplanes} of Lemma~\ref{lem:midpoint_in_CCC}, the halfspace of $H_{\Psi_1,r'}$ containing $\alpha$ is disjoint from the halfspace of $H_{\Psi_2,1}$ containing $\beta$. Thus, if $\xi \in Q^{(0)}$, then we cannot simultaneously have both $\Psi_1(\xi) \leq r'-1$ and $\Psi_2(\xi) \geq r'+1$. As a consequence, we cannot have both $z \in Z_1$ and $x \in Z_2$. This implies that $Z_1 \cap Z_2 = \emptyset$.

\mk

The conditions of Lemma~\ref{lem:join_of_contractions} are therefore met, and by applying it we deduce that $\sigma(a,b) \ge \sigma(a,c)+\sigma(c,b)-2t-2D-2E = \sigma(a,c)+\sigma(c,b)-C'_\sigma$. This completes the proof that $\sigma$ is weakly roughly geodesic with constant $C'_\sigma$.
\ep

\subsection{Coarse convexity of balls} ~

To complete the proof of Theorem~\ref{thm:properties_of_sigma}, it remains to show that balls in $(X,\sigma)$ are uniformly coarsely median-convex.

\blem \label{lem:intervals_don't_move}
There is a constant $\epsilon \geq 0$ such that for any $x,y,z\in X$ with $x\in[y,z]$, we have $d(x,\mu(x,y,z))\le \epsilon$.
\elem

\bp
According to \cite[Lemma~8.1]{bowditch_quasiflats}, there are constants $r_0$ and $r_0'$ such that $x$ lies at distance at most $r_0'$ from a point $x'$ with $d(x',\mu(x',y,z))\le r_0$. Since the coarse median $\mu$ is coarsely Lipschitz, we have $d(\mu(x,y,z),\mu(x',y,z)) \leq h(0)d(x,x')+h(0) \le h(0)(r'_0+1)$. We deduce that $d(x,\mu(x,y,z)) \le \epsilon$, where $\epsilon = r'_0+r_0+h(0)(r'_0+1)$.
\ep

\blem \label{lem:balls_median_convex}
Suppose that $(X,\mu,d)$ has quasicubical intervals and is roughly geodesic. There is a constant $M$ such that each ball in $(X,\sigma)$ is $M$--coarsely median-convex.
\elem

\bp
Fix $w \in X$ and $R \geq 0$. Let $y,z \in B_\sigma(w,R)$. Given any $a \in X$, we want to bound the distance from $x=\mu(a,y,z)$ to $B_\sigma(w,R)$.

Let $r<\sigma(w,x)$, and let $\Phi : X \ra [0,r]$ be a $K$--contraction such that $\Phi(w)=0$ and $\Phi(x)\ge r$. Lemma~\ref{lem:intervals_don't_move} tells us that $d(\mu_{x,y,z},x) \leq \epsilon$, so $|\Phi(\mu_{x,y,z})-\Phi(x)|\le\epsilon+K$. Since $\Phi$ is a $K$--quasi-median homomorphism, we have 
\begin{align*}
\mu_\R(\Phi(x),\Phi(y),\Phi(z)) 
    &\ge \Phi(\mu_{x,y,z})-K \\
    &\ge \Phi(x)-\epsilon-2K \hspace{2mm}\ge\hspace{2mm} r-\epsilon-2K.
\end{align*} 
This means that one of $\Phi(y)$ and $\Phi(z)$ must be at least $r-\epsilon-2K$, and so $\sigma(w,x) \leq \max\{\sigma(w,y),\sigma(w,z)\} + \epsilon+2K$.

This proves that $x \in B_\sigma(w,R+\epsilon+2K)$. According to Proposition~\ref{prop:coarsely_geodesic}, $\sigma$ is weakly roughly geodesic with constant $C'_\sigma$. Applying Definition~\ref{def:wrg} with $a=w$, $b=x$, and $r=\min\{R-C'_\sigma,\sigma(w,x)\}$, yields a point $x'\in B_\sigma(w,R)$ with $d(x',x)\le\epsilon+2K+3C'_\sigma=M$, which shows that balls in $(X,\sigma)$ are $M$--coarsely median-convex.
\ep

\section{Quasiconvexity and a coarse Helly property in HHSs} \label{section:hhs}

The goal of this section is to prove that \emph{hierarchically quasiconvex} subsets of hierarchically hyperbolic spaces satisfy a coarse version of the Helly property. Since coarsely median-convex subsets of a hierarchically hyperbolic space are hierarchically quasiconvex \cite[Proposition~5.11]{hhs_convexity}, this applies in particular to balls for the metric $\sigma$ constructed in Section~\ref{section:newmetric}, by Theorem~\ref{thm:properties_of_sigma}, allowing us to deduce Theorem~\ref{thm:hhs_coarsely_helly}. We also deduce the bounded packing property for hierarchically quasiconvex subgroups of groups that are HHSs.

\subsection{Background on hierarchical hyperbolicity} \label{subsection:hhsbackground} ~

Here we give a description of hierarchically hyperbolic spaces (HHSs) and hierarchically hyperbolic groups (HHGs). For full definitions, see \cite[Def.~1.1, 1.21]{behrstockhagensisto:hierarchically:2}. Briefly, an HHS consists of a quasigeodesic space $(X,d)$, a constant $E$, and a set $\s$, elements of which are called \emph{domains}. Each domain $U$ has an associated $E$--hyperbolic space $\mcC U$, and the various axioms give structure for extracting information about $X$ from these hyperbolic spaces. This includes:
\begin{shortitem}
\item 	Each domain $U$ has an associated $E$--coarsely onto, $(E,E)$--coarsely Lipschitz \emph{projection} map $\pi_U:X\to\mcC U$.
\item 	$\s$ has a partial order $\nest$, called nesting, and a symmetric relation $\bot$, called orthogonality. If $U\nest V$ and $V\bot W$, then $U\bot W$. The relations $\pnest$, $\bot$, and $=$ are mutually exclusive, and their complement, denoted $\trans$, is called transversality.
\item 	There is a bound on the size of $\pnest$--chains and pairwise orthogonal sets. 
\item 	If $U\pnest V$ or $U\trans V$ then there is a set $\rho^U_V\subset\mcC V$ of diameter at most $E$. 
\item 	If $U\pnest V$ then there is also a map $\rho^V_U:\mcC V\to\mcC U$. If $\gamma\subset\mcC V$ is a geodesic and $d_{\mcC V}(\gamma,\rho^U_V)>E$, then $\diam\rho^V_U(\gamma)\leq E$. 
\end{shortitem}

This last point is referred to as \emph{bounded geodesic image}. For $x,y\in X$, it is standard to write $d_U(x,y)$ in place of $d_{\mcC U}(\pi_U(x),\pi_U(y))$, and similarly for subsets of $X$. Moreover, we can always assume that $X$ and the associated hyperbolic spaces are graphs (for example by \cite[Lemma~3.B.6]{cornulierdelaharpe:metric}). In particular, we can and shall assume that $X$ and the $\mcC U$ are geodesic.

We say that $X$ admits an HHS structure if there is an HHS whose underlying metric space is $X$, and we write $(X,\s)$ as shorthand for the entirety of a choice of HHS structure. An HHG is a finitely generated group $G$ whose Cayley graph admits an HHS structure $(G,\s)$ such that $G$ acts cofinitely on $\s$ and elements of $G$ induce isometries $\mcC U\to\mcC gU$ for all $U\in\s$. (There are a couple of other natural regulatory assumptions that we shall not concern ourselves with here.)

\mk

The idea behind two domains being orthogonal is that one can see a direct product of associated sub-HHSs inside $X$. This is made precise by the \emph{partial realisation} axiom.

\begin{axiom}[Partial realisation]
If $\{U_i\}$ is a set of pairwise orthogonal domains, then for any choice of points $p_i\in\mcC U_i$, there is some $x\in X$ with $d_{U_i}(x,p_i)\le E$ for all $i$, and with $d_V(x,\rho^{U_i}_V)\le E$ whenever $U_i\pnest V$ or $U_i\trans V$.
\end{axiom}

In fact, one of the main tools for dealing with HHSs is the \emph{realisation theorem} \cite[Theorem~3.1]{behrstockhagensisto:hierarchically:2}, which extends the partial realisation axiom. Roughly, it says that any \emph{consistent tuple} is well-approximated by the projections of some point in $X$. In other words, performing constructions in $X$ can be reduced to performing constructions in the associated hyperbolic spaces and checking that the points produced by this process are consistent.

\begin{defi}[Consistent tuple]
For a constant $\kappa\geq E$, a tuple $(b_U)\in\prod_{U\in\s}\mcC U$ is said to be $\kappa$--consistent if 
\[
\min\big\{d_U(b_U,\rho^V_U),d_V(b_V,\rho^U_V)\big\}\leq\kappa \hspace{2mm}\text{ whenever } U\trans V, \text{ and}
\]\[
\min\big\{d_V(b_V,\rho^U_V),\diam(b_U\cup\rho^V_U(b_V))\big\}\leq\kappa \hspace{2mm}\text{ whenever } U\pnest V.
\]
\end{defi}

\begin{axiom}[Consistency]
For any $x\in X$, the tuple $(\pi_U(x))_{U\in\s}$ is $E$--consistent. 
\end{axiom}

It will be useful to be able to talk about consistency for subsets of $\s$. Given $u\in\mcC U$ and $v\in\mcC V$, we say that $u$ and $v$ \emph{satisfy the consistency inequalities} for $U$ and $V$ if 

\begin{shortitem}
\item 	$U\trans V$ and $\min\big\{d_U(u,\rho^V_U),d_V(v,\rho^U_V)\big\}\le E$, or
\item 	(after relabelling) $U\pnest V$ and $\min\big\{d_V(v,\rho^U_V),\diam(\{u\}\cup\rho^V_U(v))\big\} \le E$.
\end{shortitem}

Let us now state the realisation theorem, which will be the mechanism for our proof of Theorem~\ref{thm:hellyhqc}. We shall only need the existence part.

\begin{thm}[Realisation, {\cite[Theorem~3.1]{behrstockhagensisto:hierarchically:2}}] \label{thm:realisation}
For each $\kappa\ge E$, there are numbers $\theta_e(\kappa)$ and $\theta_u(\kappa)$ such that, if $(b_U)_{U\in\s}$ is a $\kappa$--consistent tuple, then there is some $x\in X$ with $d_U(x,b_U)\le\theta_e(\kappa)$ for all domains $U$. Moreover, the set of such $x$ has diameter at most $\theta_u(\kappa)$.
\end{thm}

A key application of the realisation theorem is for the construction of a coarse median operation for HHSs. Given three points $x,y,z$ in an HHS $(X,\s)$, let $(m_U)_{U\in\s}$ be the tuple whose $U$--entry is the median of the triple $\pi_U(x),\pi_U(y),\pi_U(z)$ in the hyperbolic space $\mcC U$. This tuple is consistent \cite[Theorem~7.3]{behrstockhagensisto:hierarchically:2}, so we define $\mu(x,y,z)$ to be a point obtained by applying the realisation theorem to the tuple $(m_U)$. (One also needs a proposition of Bowditch \cite[Proposition~10.1]{bowditch_coarse_median} to conclude that $(X,\mu,d)$ is a coarse median space.) When $X$ is an HHG, one can arrange for $\mu$ to be equivariant.

\mk

The action on the index set is what distinguishes HHGs from groups that are HHSs, and this turns out to be an important distinction. For example, the property of being an HHS is invariant under quasi-isometries, but there are groups that are virtually HHGs but not HHGs themselves. Indeed, the $(3,3,3)$ triangle group is virtually abelian, but, as mentioned in the introduction, it is not coarsely injective \cite{hoda:crystallographic}, and it therefore cannot be an HHG by Corollary~\ref{cor:list}. A more direct proof not relying on the results of this paper is given in \cite{petytspriano:unbounded}. On the other hand, any group that is an HHS can be equipped with a coarse median \cite{behrstockhagensisto:hierarchically:2}, but this is may fail to be equivariant if the structure is only an HHS structure.

A related notion that is closed under taking subgroups is that of a group that acts on an HHS $(X,\s)$ by \emph{HHS automorphisms}. In other words, it acts on $X$ isometrically, and on $\s$ with the regulatory assumptions alluded to above, but the action on $\s$ need not be cofinite. The median is still equivariant for such actions. 

\mk

In the theory of hyperbolic spaces, an important class of subsets are the quasiconvex subsets, because they inherit the structure of the ambient space. The natural analogue in the setting of hierarchical hyperbolicity is that of a \emph{hierarchically quasiconvex} subset.

\begin{defi}[Hierarchical quasiconvexity]
A subset $Y$ of an HHS $(X,\s)$ is said to be hierarchically quasiconvex if there is a function $k$ such that: every $\pi_U(Y)$ is $k(0)$--quasiconvex; and if $x\in X$ has $d_U(x,Y)\leq r$ for all $U\in\s$, then $d_X(x,Y)\leq k(r)$. 
\end{defi}

We finish this section with some examples.

All hyperbolic groups are hierarchically hyperbolic, as are the (extended) mapping class groups of finite type surfaces \cite{behrstockhagensisto:hierarchically:1}; Teichm\"uller space with either of the standard metrics \cite{behrstockhagensisto:hierarchically:1}; many graphs defined from curves on surfaces, including the pants graph \cite{vokes:hierarchical}; quotients of mapping class groups by powers of pseudo-Anosovs \cite{behrstockhagensisto:asymptotic} and Dehn-twist subgroups \cite{behrstockhagenmartinsisto:combinatorial}; extensions of Veech groups \cite{dowdalldurhamleiningersisto:extensions}; the genus-two handlebody group \cite{miller:stable}; fundamental groups of closed 3--manifolds without $\Nil$ or $\Sol$ components \cite{behrstockhagensisto:hierarchically:2}; right angled Artin groups \cite{behrstockhagensisto:hierarchically:1}; and in fact all known cubical groups \cite{hagensusse:onhierarchical}. Aside from the extensions of Veech groups and some 3--manifold groups, the groups listed here are all known to be HHGs, not merely HHSs.

There are also various ways to combine HHSs and HHGs to produce new ones. For example, both classes are closed under relative hyperbolicity \cite{behrstockhagensisto:hierarchically:2}; any graph product of HHGs is an HHG \cite{berlynerussell:hierarchical}; and many graphs of groups are HHGs \cite{behrstockhagensisto:hierarchically:2,berlairobbio:refined,robbiospriano:hierarchical}.

\subsection{Coarse injectivity} ~

Here we prove our result on hierarchically quasiconvex subsets of an HHS and deduce that HHSs are coarsely injective when equipped with the metric $\sigma$ from Section~\ref{section:newmetric}. We then deduce that every HHG acts properly cocompactly by isometries on a coarsely injective space.

We shall make use of the following powerful result for hyperbolic spaces. The version stated here is a combination of \cite[Lemma~5.1]{chepoidraganvaxes:core} and the proof of \cite[Theorem~5.1]{chepoidraganvaxes:core}. It states in particular that quasiconvex subsets of a hyperbolic graph satisfy a coarse version of the Helly property. Throughout this section, we say that subsets $Z_1$ and $Z_2$ of a metric space $(X,d)$ are $r$--close if there exist $z_1\in Z_1$ and $z_2\in Z_2$ with $d(z_1,z_2)\le r$.

\begin{thm}[{\cite{chepoidraganvaxes:core}}] \label{thm:cdvhelly}
The following holds for any nonnegative constants $E$, $r$, and $k_0$. Let $Y$ be an $E$--hyperbolic graph and let $y$ be a vertex of $Y$. Suppose that $\mathcal Q$ is a collection of pairwise $2Er$--close $k_0$--quasiconvex subsets of $Y^{(0)}$ with the property that $\{d(y,Q):Q\in\mathcal Q\}$ is bounded. By discreteness, we can fix $Q\in\mathcal Q$ with $d(y,Q)$ maximal. Let $z\in Q$ have $d(y,z)=d(y,Q)$, and let $c$ be the point on a geodesic $[y,z]$ with $d(c,z)=\min\{Er,d(y,z)\}$. Then $d(c,Q')\leq r'$ for all $Q'\in\mathcal Q$, where $r'=\max\{2k_0+5E,Er+k_0+3E\}$.
\end{thm}

The strength of this theorem is twofold. Firstly, the constant $r'$ is independent of the size of the set $\mathcal Q$---a statement with this independence does not seem to appear elsewhere in the geometric group theory literature. The second strength is that the construction of the point $c$ is both completely explicit and allows for a lot of flexibility in the choice of $y$. Observe that the condition that $\{d(y,Q):Q\in\mathcal Q\}$ is bounded is satisfied automatically if any $Q\in\mathcal Q$ is bounded.

We will now prove that hierarchically quasiconvex subsets of a hierarchically hyperbolic space satisfy a coarse version of the Helly property.

\begin{thm}[Coarse Helly property] \label{thm:hellyhqc}
Let $(X,\s)$ be an HHS with constant $E$, and let $\mathcal Q$ be a collection of $k$--hierarchically quasiconvex subsets of $X$ such that either $\mathcal Q$ is finite or $\mathcal Q$ contains an element with bounded diameter. Suppose that there is a constant $r$ such that any two elements of $\mathcal Q$ are $r$--close. There is a constant $R=R(E,k,r)$ such that there is a point $x\in X$ with $d(x,Q)\leq R$ for all $Q\in\mathcal Q$.
\end{thm}

\begin{proof}	
Let us say that a domain $U$ \emph{begets} a domain $V$ if either $U\trans V$ or $U\pnest V$. If $U$ begets $V$ then there is a well-defined bounded set $\rho^U_V$.
	
Let $\mathcal U=\{U_1,\dots,U_n\}$ be a maximal collection of pairwise orthogonal, nest-minimal domains. Note that we may choose $\mathcal U$ arbitrarily. For any domain $V\in\s\smallsetminus\mathcal U$, there is some $i$ such that $U_i$ begets $V$. By \cite[Lemma~1.5]{durhamhagensisto:boundaries}, for any domain $V\in\s$ we have $d_V(\rho^{U_i}_V,\rho^{U_j}_V)\leq2E$ whenever $U_i$ and $U_j$ both beget $V$. Moreover, recall that $\diam\rho^{U_i}_V\leq E$. At the cost of increasing the hierarchical hyperbolicity constant to at most $10E$, we can therefore perturb the HHS structure to assume that every $\rho^{U_i}_V$ is a singleton, and that $\rho^{U_i}_V=\rho^{U_j}_V$ whenever both $U_i$ and $U_j$ beget $V$. We write $\rho^\mcU_V$ for the singleton 
\[ 
\rho^\mcU_V\hspace{2mm}=\bigcup_{\{i\hspace{1mm}:\hspace{1mm}U_i\text{ begets }V\}}\rho^{U_i}_V.
\]
As mentioned, the construction of $\mcU$ ensures that the point $\rho^\mcU_V$ exists for all $V\in\s\smallsetminus\mcU$.
	
We are free to assume that if $r>0$ then $r>1$. Thus, by definition of hierarchical quasiconvexity and the fact that projection maps are $(E,E)$--coarsely Lipschitz, we have that, for any domain $V$, the sets $\pi_V(Q)$, for 
$Q \in \mathcal Q$, are pairwise $2Er$--close and $k_0$--quasiconvex, where $k_0=k(0)$. We assumed that either $\mathcal Q$ is finite or it contains an element with bounded diameter, so for any point $y \in X$ and any domain $V$, the set $\{d_V(y,Q):Q\in\mathcal Q\}$ is bounded. Let $r'$ be as in the statement of Theorem~\ref{thm:cdvhelly}. That theorem now allows us to choose, for each $U\in\mathcal U$, a point $b_U$ in $\mcC U$ with $d_U(b_U,Q)\leq r'$ for all $Q\in\mathcal Q$. For any other domain $V$, let $b_V$ be the point of $\mcC V$ obtained by applying Theorem~\ref{thm:cdvhelly} in the hyperbolic graph $\mcC V$, with quasiconvex subsets $\{\pi_V(Q):Q\in\mathcal Q\}$ and starting vertex~$\rho^\mcU_V$.
	
\mk

{\bf Claim:} The tuple $(b_V)_{V\in\s}$ is $(r'+7E+Er)$--consistent.

\mk
	
{\bf Proof of Claim:} Suppose that $W$ begets $V$, and that $d_V(\rho^W_V,\rho^\mcU_V)\leq2E$. Assume that $d_V(b_V,\rho^W_V)>r'+7E+Er$. By the construction of $b_V$, there exists some $Q\in\mathcal Q$ such that $d_V(b_V,Q) \leq Er$. As a consequence, we have $d_V(Q,\rho^W_V) \ge d_V(b_V,\rho^W_V) - d_V(b_V,Q) -\diam\rho^W_V > r'+6E$. If $W\trans V$, then $\pi_W(Q)$ is contained in the $E$--neighbourhood of $\rho^V_W$ by consistency for elements of $Q$. In particular, $d_W(\rho^V_W,b_W)\leq r'+E$ as $b_W$ is $r'$--close to $\pi_W(Q)$. If $W\pnest V$, then since $\pi_V(Q)$ is $k_0$--quasiconvex and $r'+6E>k_0+6E$, bounded geodesic image and consistency show that the set $\rho^V_W(\pi_V(Q))$ has diameter at most $E$, and its $E$--neighbourhood contains $\pi_W(Q)$. Moreover, its $E$--neighbourhood contains $\rho^V_W(b_V)$ by bounded geodesic image, as witnessed by the geodesic used to construct $b_V$. Thus
\begin{align*}
\diam(b_W\cup\rho^V_W(b_V))
	& \leq d_W(b_W,Q)+\diam\pi_W(Q)+d_W(Q,\rho^V_W(b_V))+\diam\rho^V_W(b_V) \\
	& \leq r'+3E+3E+E=r'+7E.
\end{align*}
		
The above paragraph will be referred to as $(*)$ for the rest of the proof of the claim. We split the checking of the consistency inequalities for pairs $(V,W)$ of domains into three cases.
		
\medskip
\noindent\underline{Case~1}. $W\in\mathcal U$ begets $V$.
		
In this case, $\rho^W_V=\rho^\mcU_V$, so we are done by $(*)$.
		
\medskip
\noindent\underline{Case~2}. There is some $U\in\mcU$ that begets both $V$ and $W$.
		
Proposition~1.8 of \cite{behrstockhagensisto:hierarchically:2} states that if $W$ begets $V$ then $\rho^U_V$ and $\rho^U_W$ satisfy the consistency inequalities for $V$ and $W$. Consequently, by $(*)$, the only case we need to check here is when $U\trans W$, $W\pnest V$, and $\diam(\rho^U_W\cup\rho^V_W(\rho^U_V))\leq2E$. Assuming that $d_V(\rho^W_V,b_V)>r'+7E+Er$, there are two possibilities, depending on the location of $\rho^U_V$. 
		
If there is a geodesic $[\rho^U_V,b_V]$ that is disjoint from the $E$--neighbourhood of $\rho^W_V$, then $\diam(\rho^V_W(\rho^U_V)\cup\rho^V_W(b_V))\leq E$, so $d_W(\rho^U_W,\rho^V_W(b_V))\leq3E$. Moreover, for each $Q\in\mathcal Q$ there is some $q\in Q$ such that any geodesic $[b_V,\pi_V(q)]$ is disjoint from the $E$--neighbourhood of $\rho^W_V$. In particular, $\rho^V_W(b_V)$ is $2E$--close to each $\pi_W(q)$, and hence $\rho^U_W$ is $5E$--close to each $\pi_W(Q)$. Since $b_W$ lies on a shortest geodesic between $\rho^U_W$ and some $\pi_W(Q)$, we get that $d_W(b_W,\rho^U_W)\leq5E$, and so $b_W$ is $8E$--close to $\rho^V_W(b_V)$.
		
Otherwise, every geodesic $[\rho^U_V,b_V]$ meets the $E$--neighbourhood of $\rho^W_V$. By construction of $b_V$, there exists $Q\in\mathcal Q$ such that $d_V(\rho^W_V,Q)> (r'+7E+Er) +Er-2E = r'+5E+2Er$. By the same argument as in $(*)$, we now get that $\rho^V_W(b_V)$ is $3E$--close to $\pi_W(Q)$, which has diameter at most $3E$. Hence $\diam(b_W\cup\rho^V_W(b_V))\leq r'+7E$.
		
\medskip
\noindent\underline{Case 3.} No $U_i$ begets both $V$ and $W$, and neither $V$ nor $W$ is in $\mcU$.
		
After relabelling we can assume that $U_1$ begets $V$ and $U_2$ begets $W$. Since $U_1$ does not beget $W$ we have $U_1\bot W$, and similarly $U_2\bot V$. In particular, the only case that needs checking is when $V\trans W$. The partial realisation axiom applied to any points $p_1 \in \mcC U_1$ and $p_2 \in \mcC U_2$ provides a point $z\in X$ such that $d_V(z,\rho^{U_1}_V)\leq E$ and $d_W(z,\rho^{U_2}_W)\leq E$. By consistency for $z$, we have that either $d_V(\rho^W_V,\rho^{U_1}_V)\leq2E$ or $d_W(\rho^V_W,\rho^{U_2}_W)\le2E$. We are done by $(*).\hfill\diamondsuit$

\mk
	
In light of the claim, Theorem~\ref{thm:realisation} provides a point $x\in X$ such that $d_V(x,b_V)\leq\theta_e(r'+7E+Er)$ for all $V\in\s$. By construction of the points $b_V$, we have that $d_V(x,Q)\leq r'+\theta_e(r'+7E+Er)$ for all $Q\in\mathcal Q$. Hierarchical quasiconvexity of the $Q$ now tells us that $x$ is $k(r'+\theta_e(r'+7E+Er))$--close to $Q$ for all $Q\in\mathcal Q$. 
\end{proof}

It is worth noting that the proof of Theorem~\ref{thm:hellyhqc} gives flexibility of a similar kind to that in Theorem~\ref{thm:cdvhelly}. Indeed, we are free in our choice of $\mathcal U$, and once this is chosen we apply the Chepoi--Dragan--Vax\`es construction in each of the hyperbolic spaces associated with $\mathcal U$, without restriction on the choice of starting point therein. We shall not need to make use of this in the present paper.

\begin{cor} \label{cor:hhscoarselyhelly}
If $X$ is an HHS, then $(X,\sigma)$ is coarsely injective, hence roughly geodesic.
\end{cor}

\begin{proof}
By Proposition~\ref{prop:applies_to_hhss}, the geodesic coarse median space $(X,\mu,d)$ has quasicubical intervals, so Theorem~\ref{thm:properties_of_sigma} tells us that the metric $\sigma$ is weakly roughly geodesic on $X$, that it is quasi-isometric to $d$, and that $\sigma$--balls are uniformly coarsely median--convex. Let $\{B_\sigma(x_i,r_i):i\in I\}$ be a family of balls in $(X,\sigma)$ with the property that $\sigma(x_i,x_j)\le r_i+r_j$ for all $i,j\in I$. Since $\sigma$ is weakly roughly geodesic, there is a constant $\delta$, independent of the family of balls, such that the balls $B_\sigma(x_i,r_i+\delta)$ intersect pairwise.
	
Let $B_i$ be the image of the ball $B_\sigma(x_i,r_i+\delta)$ under the identity quasi-isometry $(X,\sigma)\to(X,d)$. The $B_i$ are uniformly coarsely median-convex, and so they are uniformly hierarchically quasiconvex by \cite[Proposition~5.11]{hhs_convexity}. They also intersect pairwise, and each is bounded, so Theorem~\ref{thm:hellyhqc} produces a point at uniformly bounded $d$--distance from each $B_i$. As $d$ and $\sigma$ are quasi-isometric, this point is at uniformly bounded $\sigma$--distance from each $B_\sigma(x_i,r_i+\delta)$. Thus $(X,\sigma)$ is coarsely injective.

Since any injective space is geodesic, we deduce that the coarsely injective metric space $(X,\sigma)$ is not merely weakly roughly geodesic, but actually roughly geodesic, as it is coarsely dense in its injective hull.
\end{proof}

Usually it really is necessary to change the metric: \cite[Ex.~5.13]{chalopinchepoigenevoishiraiosajda:helly} shows that $\Z^3$ with the standard $\ell^1$ metric is not coarsely injective, though it is an HHG.

\mk

We now explain how to deduce the existence of a bicombing from work of Lang. See Section~\ref{subsection:metric_consequences} for the definitions of roughly conical and roughly reversible bicombings.

\bcor \label{cor:hhs_bicombing}
If $(X,\s)$ is an HHS, then $(X,\sigma)$ admits a roughly conical, roughly reversible, bicombing by rough geodesics that is coarsely equivariant under the automorphism group of $(X,\s)$.
\ecor

\bp
According to Corollary~\ref{cor:hhscoarselyhelly}, the metric space $(X,\sigma)$ is coarsely injective, so it is $D$--coarsely dense in its injective hull for some $D$. A construction of Lang shows that every injective metric space $E$ admits a conical, reversible, geodesic, $\Isom E$--invariant bicombing $\gamma'$ \cite{lang:injective}. Take $E=E((X,\sigma))$. For each $a,b \in X$ and $t \in [0,1]$, define $\gamma_{a,b}(t)$ as any point of $X$ at distance at most $D$ from $\gamma'_{a,b}(t)$. Since $\gamma(t)$ is at uniform distance $D$ from $\gamma'(t)$, we deduce that $\gamma$ is a bicombing on $(X,\sigma)$ with the listed properties.
\ep

Note that if the action of the automorphism group of $(X,\s)$ on $X$ is free, then the bicombing may be chosen to be actually equivariant.

\mk

Let us now discuss the consequences of our construction for HHGs.

\begin{cor} \label{cor:inducedaction}
If $G$ is an HHG, then $G$ admits a proper, cocompact, isometric action on the coarsely injective space $(G,\sigma)$.
\end{cor}

\begin{proof}
$(G,\sigma)$ is coarsely injective by Corollary~\ref{cor:hhscoarselyhelly}. Since the median is equivariant in an HHG, Lemma~\ref{lem:acting_on_new_metric} tells us that the action is isometric. Properness and cocompactness follow from Proposition~\ref{prop:bilipschitzmetrics}. 
\end{proof}

\begin{rmk} \label{remark:HHS_aut_is_enough}
In fact, we do not quite need to assume that we have a hierarchically hyperbolic group in Corollary~\ref{cor:inducedaction}: we only need a proper cocompact action by median isometries on an HHS. In fact, cocompactness can be relaxed to coboundedness for the sake of the applications in this paper. For example, it would be sufficient to assume that $G$ is a group acting properly coboundedly by HHS automorphisms on an HHS. The consequences for HHGs listed here and in the introduction therefore apply in this generality.
\end{rmk}

The next lemma is a modified version of \cite[Proposition~6.7]{chalopinchepoigenevoishiraiosajda:helly}, in which the assumption that the hull is proper has been dropped.

\begin{lem} \label{lem:modifiedccgho}
If a group $G$ acts properly coboundedly on a coarsely injective space $X$, then $G$ acts properly coboundedly on the injective hull $E(X)$. In particular, every HHG admits a proper, cobounded action on an injective space.
\end{lem}

\begin{proof}
There is an induced action of $G$ on $E(X)$ and the isometric embedding $e \colon X \to E(X)$ is equivariant with respect to this induced action \cite[Proposition~3.7]{lang:injective}.  To simplify notation, we identify the points of $X$ with their images under $e$ and thus identify $X$ with $E(X)$.
The Hausdorff distance between $X$ and $E(X)$ is bounded by some constant $D$, so the action of $G$ on $E(X)$ is cobounded. For properness, let $Y\subset E(X)$ be bounded and let $Y'=\{x\in X:d(Y,x)\leq D\}\neq\varnothing$. Since $e$ is an isometric embedding, $Y'$ is bounded. If $g\in G$ has $gY\cap Y\neq\varnothing$, then pick $y\in Y$ with $gy\in Y$ and let $x\in X$ have $d(y,x)\leq D$. Then $d(gy,gx)\leq D$, so $gx$ is $D$--close to $Y$. That is, $gY'\cap Y'\neq\varnothing$, so since $Y'$ is bounded and the action of $G$ on $X$ is proper, there are only finitely many such $g$. The final sentence follows from Corollary~\ref{cor:inducedaction}.
\end{proof}

Next we strengthen Corollary~\ref{cor:hhs_bicombing} in the case of HHGs. In particular, this applies to (extended) mapping class groups of finite type surfaces.

\begin{cor} \label{semihyperbolicity}
If $G$ is an HHG, then $G$ is semihyperbolic.
\end{cor}

\begin{proof}
By Lemma~\ref{lem:modifiedccgho}, $G$ acts properly coboundedly on an injective space $E$. Every orbit map $G\to E$ is a $G$--equivariant quasi-isometry. By \cite[Proposition~3.8]{lang:injective}, $E$ has a $G$--invariant, bounded, geodesic bicombing in the sense of \cite{alonsobridson:semihyperbolic}. As the action of $G$ on itself is free, it is semihyperbolic by \cite[Theorem~4.1]{alonsobridson:semihyperbolic}.
\end{proof}

\begin{cor} \label{prop:torsion}
If $G$ is an HHG, then $G$ has finitely many conjugacy classes of finite subgroups. 
\end{cor}	

\begin{proof}
By Lemma~\ref{lem:modifiedccgho}, $G$ acts properly coboundedly on an injective space $E$. Let $x\in E$ and let $r$ be a constant such that $G\cdot x$ is $r$--coarsely dense in $E$. Let $F$ be a finite subgroup of $G$. By \cite[Proposition~1.2]{lang:injective}, there is a point $z\in E$ that is fixed by $F$, and hence $F$ fixes the ball $B(z,r)$ in $E$, which contains a point of $G\cdot x$. It follows that a conjugate of $F$ fixes a point in $B(x,r)$, and we are done by properness of the action.
\end{proof}

\subsection{Packing subgroups} \label{subsection:packing} ~

Here we describe the application to bounded packing mentioned in the introduction. Following Hruska and Wise \cite{hruskawise:packing}, we say that a finite collection $\mathcal H$ of subgroups of a discrete group $G$ has \emph{bounded packing in $G$} if for each $N$ there is a constant $r$ such that for any collection of $N$ distinct cosets of elements of $\mathcal H$, at least two are separated by a distance of at least $r$ (with respect to some left-invariant, proper distance). If $\mathcal H$ consists of a single subgroup $H$, then we say that $H$ has bounded packing in $G$. 

\begin{cor} \label{cor:boundedpacking}
If $\mathcal H$ is a finite collection of hierarchically quasiconvex subgroups of a group $G$ that is an HHS, then $\mathcal H$ has bounded packing in $G$.
\end{cor}

\begin{proof}
By Theorem~\ref{thm:hellyhqc}, any finite collection of cosets of elements of $\mathcal H$ that are pairwise $r$--close must all come $R$--close to a single point $x\in G$. In other words, they all intersect the $R$--ball about $x$. Since distinct cosets of a given subgroup are disjoint and balls in $G$ are finite, this bounds the size of the collection of cosets.
\end{proof}

In the case of quasiconvex subgroups of hyperbolic groups, one can use Theorem~\ref{thm:cdvhelly} in place of Theorem~\ref{thm:hellyhqc} in this argument to provide a new, simpler proof of bounded packing. This type of argument is also implicit in \cite[Remark~4.4, Corollary~4.5]{hagenpetyt:projection}, though the coarse Helly property for quasiconvex subgroups of hyperbolic groups is established in a much less efficient way there. 

Previous proofs of this result work by induction on the height of subgroups. However, this line of reasoning does not generalise outside the setting of strict negative curvature; indeed, no subgroup of a flat can ever have finite height. Moreover, Theorems~\ref{thm:cdvhelly} and~\ref{thm:hellyhqc} are purely geometric: there is no group action involved. It therefore seems that the most natural way to establish bounded packing for quasiconvex subgroups of hyperbolic groups is via the Chepoi--Dragan--Vax\`es theorem as described above.

\mk

If a group $G$ has a codimension--1 subgroup $H$, then Sageev's construction yields an action of $G$ on a CAT(0) cube complex, and if the conjugates of $H$ satisfy the coarse Helly property, then it follows that the action of $G$ on the CAT(0) cube complex is cocompact (\cite{sageev:codimension}). This raises the following question.

\begin{question}
Does the mapping class group have property $FW_\infty$, i.e. does any action of the mapping class group on a finite-dimensional CAT(0) cube complex have a fixed point?
\end{question}

Note that property $FW_\infty$ is intermediate between having no virtual surjection onto $\Z$ and Kazhdan's property (T). There are known restrictions on what an action of the mapping class group on a CAT(0) cube complex could look like. Indeed, the mapping class group of a surface of genus at least three does not admit a properly discontinuous action by semisimple isometries on a complete CAT(0) space (\cite{kapovichleeb:actions,bridsonhaefliger:metric,bridson:semisimple}), nor, more specifically, does it act properly on a CAT(0) cube complex (even an infinite dimensional one) \cite{genevois:cubical}.

More generally, in relationship with property (T) and the Haagerup property, the existence of non-trivial actions of the mapping class group on various generalisations of CAT(0) cube complexes remains mysterious; for example median spaces, Hilbert spaces, CAT(0) spaces, and $L^p$ spaces. The coarse version of the Helly property established here may prove useful in the study of such actions.

\section{Strong Shortcut Property} \label{section:shortcut}

In this section we will prove that coarsely injective spaces of uniformly bounded geometry are strongly shortcut.  Recall that a metric space has \emph{uniformly bounded geometry} if, for any $r > 0$, there exists a uniform $N(r) \in \N$ such that every ball of radius $r$ contains at most $N(r)$ points.

A \emph{Riemannian circle} $S$ is $S^1$ endowed with a geodesic metric of some length $|S|$.  A roughly geodesic metric space $(X,\sigma)$ is \emph{strongly shortcut} if there exists $K > 1$ such that for any $C > 0$ there is a bound on the lengths $|S|$ of $(K,C)$--quasi-isometric embeddings $S \to X$ of Riemannian circles $S$ in $(X,\sigma)$ \cite{Hoda:shortcut_spaces}.  A group $G$ is \emph{strongly shortcut} if it acts properly and coboundedly on a strongly shortcut metric space \cite{Hoda:shortcut_graphs, Hoda:shortcut_spaces}.

We will now give a brief description of the injective hull construction of Isbell \cite{isbell:six}, which was later rediscovered by Dress \cite{dress:trees} and Chrobak and Larmore
\cite{chrobaklarmore:generosity}.  For a nice discussion on this construction, see Lang \cite{lang:injective}.
Let $(X,\sigma)$ be a metric space.  A \emph{radius function} on $X$ is a function $f \colon X \to \R_{\ge 0}$ for which \[ \sigma(x,y) \le f(x) + f(y) \] for every $x,y \in X$.  A radius function $f \colon A \to \R_{\ge 0}$ on any subspace of $A \subseteq X$ is called a \emph{partial radius function} on $X$.  If $f,g \colon X \to \R_{\ge 0}$ are two radius functions then $f$ \emph{dominates} $g$ if $f(x) \ge g(x)$ for all $x \in X$.  A radius function $f \colon X \to \R_{\ge 0}$ is minimal if the only radius function it dominates is itself.

If $f \colon A \to \R_{\ge 0}$ is a partial radius function on $X$ then there exists a minimal radius function $g \colon X \to \R_{\ge 0}$ such that $g|_A$ is dominated by $f$.  For any $x \in X$, the function $\sigma(\cdot,x)$ is a minimal radius function.  If $f,g \colon X \to \R_{\ge 0}$ are two minimal radius functions then \[ |f-g|_{\infty} = \sup_{x \in X}\bigl|f(x) - g(x)\bigr|\] is finite.  The set of minimal radius functions on $X$, with metric given by $d_{E(X)}(f,g) = |f-g|_{\infty}$, is the \emph{injective hull} $E(X)$ of $X$.  The isometric embedding $e \colon X \hookrightarrow E(X)$ sends $x \in X$ to the minimal radius function $e(x) \colon y \mapsto \sigma(x,y)$ and for any $x \in X$ and $f \in E(X)$ we have $d_{E(X)}\bigl(e(x),f\bigr) = f(x)$.

\begin{lem}
  \label{lem:dom_distance} Let $(X,\sigma)$ be a metric space.
  Let $g \colon X \to \R_{\ge 0}$ be a minimal radius function,
  let $\bar f \colon X \to \R_{\ge 0}$ be a radius function and
  let $f \colon X \to \R_{\ge 0}$ be any minimal radius function
  dominated by $\bar f$.  Then
  $|g - f|_{\infty} \le |g - \bar f|_{\infty}$.
\end{lem}
\begin{proof}
  Let $y \in X$.  Then
  $f(y) \le \bar f(y) \le g(y) + |g - \bar f|_{\infty}$ and so
  $f(y) - g(y) \le |g - \bar f|_{\infty}$.  It remains to prove that
  $g(y) - f(y) \le |g - \bar f|_{\infty}$.  By minimality of $g$, for
  any $\epsilon > 0$, there exists $z \in X$ for which
  $g(y) + g(z) < \sigma(y,z) + \epsilon$.  Then, since $f$ is a radius
  function dominated by $\bar f$, we have
  \begin{align*}
    f(y) &\ge \sigma(y,z) - f(z) \\
    &\ge \sigma(y,z) - \bar f(z) \\
    &\ge \sigma(y,z) - g(z) - |g - \bar f|_{\infty} \\
    &> g(y) - \epsilon - |g - \bar f|_{\infty}
  \end{align*}
  and so $g(y) - f(y) < |g - \bar f|_{\infty} + \epsilon$ which
  completes the proof since we chose $\epsilon > 0$ arbitrarily.
\end{proof}

\begin{thm} \label{thm:coarse_helly_implies_shortcut}
  Let $(X,\sigma)$ be a coarsely injective metric space.  If
  $(X,\sigma)$ has uniformly bounded geometry then
  $(X,\sigma)$ is strongly shortcut.
\end{thm}
\begin{proof}
  In order to prove this theorem, we will show that for some uniform radius $r$, a $(K,C)$-quasi-isometric embedding of a Riemannian circle $S \to X$ implies the existence of a ``center'' point $x$ such that the cardinality of the ball $B(x,r)$ is bounded below by an expression that tends to infinity as $K$ approaches $1$ and $|S|$ approaches infinity.  If $X$ is not strongly shortcut then, for any $K > 1$, it will admit $(K,C_K)$--quasi-isometric embeddings of arbitrarily long Riemannian circles so that we would then contradict the uniformly bounded geometry assumption.
  
  Let $X \to E(X)$ be the embedding of $(X,\sigma)$
  into its injective hull and view this embedding as an inclusion of a
  subspace.  By Proposition~\ref{prop:dense_in_hull}, the subspace $X$ is $\delta$--coarsely dense in $E(X)$ for some $\delta > 0$.  So
  there is a retraction $r \colon E(X) \to X$ such that
  $r$ is a $(1,2\delta)$--quasi-isometry.
  
  Let $\phi \colon S \to X$ be a $(K,C)$--quasi-isometric
  embedding of a Riemannian circle.  Let
  $f'' \colon \phi(S) \to \R_{\ge 0}$ be the constant function taking
  the value $K\frac{|S|}{4}+C$.  Then $f''$ is a radius function on $\phi(S) \subset X$.  Let
  $f' \colon \phi(S) \to \R_{\ge 0}$ be a minimal
    radius function on $\phi(S)$ dominated by $f''$.  Then for each
  $x \in \phi(S)$ and each $\epsilon > 0$, there exists a
  $y \in \phi(S)$ for which $f'(x) + f'(y) < \sigma(x,y) + \epsilon$.
  Since $f'$ is a partial radius function on
  $X$ we can let $f \colon X \to \R_{\ge 0}$ be a
  minimal radius function on $X$ dominated
    by $f'$.  Then $f$ is a point of $E(X)$, by definition of $E(X)$.  Moreover, if $x \in \phi(S)$ then $d_{E(X)}\bigl(f,x) = f(x) \le f'(x) \le f''(x) = K\frac{|S|}{4}+C$ so if $\bar s$ is the antipode in $S$ of any element of $s \in \phi^{-1}(x)$ then we have
    \[ d_{E(X)}\bigl(x, \phi(\bar s)\bigr)
    = d_{E(X)}\bigl(\phi(s), \phi(\bar s)\bigr) \ge \frac{1}{K} d_S(s, \bar s) - C = \frac{|S|}{2K} - C \]
    and
    \[ d_{E(X)}\bigl(x, \phi(\bar s)\bigr) \le d_{E(X)}\bigl(x,f\bigr) + d_{E(X)}\bigl(f,\phi(\bar s)\bigr)
    = f(x) + f\bigl(\phi(\bar s)\bigr) \le f(x) + K\frac{|S|}{4} + C \]
    so that $f(x) \ge \frac{|S|}{2K} - K\frac{|S|}{4} - 2C = \frac{2-K^2}{4K}|S| - 2C$.
    Thus we have shown that
  \[ \frac{2-K^2}{4K}|S| - 2C \le f(x) \le K\frac{|S|}{4}+C \]
  for any $x \in \phi(S)$.

  For $x,y \in X$ let $\ell_{x,y} = f(x) + f(y) - \sigma(x,y)$.
  Since $f$ is dominated by $f'$ and $f'$ is a minimal radius function on $\phi(S)$,
  for each $x \in \phi(S)$ and each $\epsilon > 0$, there exists
  $y \in \phi(S)$ such that $\ell_{x,y} < \epsilon$.  Moreover, for
  $a,b \in S$ we have
  \begin{align*}
    \frac{2-K^2}{2K}|S| - 4C
    &\le f\bigl(\phi(a)\bigr) + f\bigl(\phi(b)\bigr) \\
    & = \sigma\bigl(\phi(a),\phi(b)\bigr) + \ell_{\phi(a),\phi(b)} \\
    & \le Kd_S(a,b) + C + \ell_{\phi(a),\phi(b)}
  \end{align*}
  and so
  $d_S(a,b) \ge \frac{2-K^2}{2K^2}|S| - \frac{\ell_{\phi(a),\phi(b)} +
    5C}{K}$.

\mk

  {\bf Claim:}
  Let $x \in \phi(S)$.  There exists a sequence of minimal radius
  functions $(f_x^k \colon X \to \R_{\ge 0} )_k$ where $k$
  ranges in $\{0,1,\ldots,M_x\}$ such that 
  $M_x = \bigl\lfloor\frac{f(x)}{\delta}\bigr\rfloor$ and the
  following properties hold for all $k$, $k'$ and $y$.
  \begin{enumerate}
  \item \label{itm:init} $f_x^0 = f$
  \item \label{itm:isom} $d_{E(X)}(f_x^k,f_x^{k'}) = \delta|k - k'|$
  \item \label{itm:ineq}
    $f(y) + k \delta - \ell_{x,y} \le f_x^k(y) \le f(y) +
    \max\{0,k\delta - \ell_{x,y}\}$
  \end{enumerate}
  \mk

  {\bf Proof of Claim:}
  We construct the $(f_x^k)_k$ by induction on $k$.  By Property~(\ref{itm:init}), we must start with $f_x^0 = f$.  Assuming we have
  $f_x^{k-1}$, we will begin by defining a radius function
  $\bar f_x^k$.  Set $\bar f_x^k(x) = f_x^{k-1}(x) - \delta$.  By
  minimality of $f_x^{k-1}$, there exists $y \in X$ for which the inequality
  \begin{equation}\label{eqn:tight} f_x^{k-1}(y) + f_x^{k-1}(x) - \delta < \sigma(x,y)\end{equation} holds.  Indeed, if no such $y$ existed then \[ y \mapsto
  \begin{cases}f_x^{k-1}(y) & \text{if $y \neq x$} \\
  f_x^{k-1}(y) - \frac{\delta}{2} & \text{if $y = x$}\end{cases}
   \] would be a radius function that is dominated by but not equal to $f_x^{k-1}$ and this would contradict minimality of $f_x^{k-1}$.  Set $\bar f_x^k(x) = f_x^{k-1}(x) - \delta$.  For all
  $y \in X\setminus \{x\}$ satisfying Equation~(\ref{eqn:tight}), set
  $\bar f_x^k(y) = \sigma(x,y) - f_x^{k-1}(x) + \delta$.  For all
  other $y \in X\setminus \{x\}$, set $\bar f_x^k(y) = f_x^{k-1}(y)$.  Then, except for at $y=x$, we have $\bar f_x^k(y) \ge f_x^{k-1}(y)$.  Thus, to check that $\bar f_x^k$ is a radius function, we need only verify that $\bar f_x^k(x) + \bar f_x^k(y) \ge \sigma(x,y)$ for any $y \in X$.  When $y = x$ the inequality $\bar f_x^k(x) + \bar f_x^k(y) \ge \sigma(x,y)$ is equivalent to $f_x^{k-1}(x) \ge \delta$, which holds by the inductive application of Property~(\ref{itm:isom}) and the triangle inequality.  When $y$ satisfies Equation~(\ref{eqn:tight}) the inequality $\bar f_x^k(x) + \bar f_x^k(y) \ge \sigma(x,y)$ is equivalent to $f_x^{k-1}(x) - \delta + \sigma(x,y) - f_x^{k-1}(x) + \delta \ge \sigma(x,y)$, which holds with equality.  Finally, when $y$ does not satisfy Equation~(\ref{eqn:tight}), we have $\bar f_x^k(x) + \bar f_x^k(y) = f_x^{k-1}(x) - \delta + f_x^{k-1}(y) \ge f_x^{k-1}(x) - \delta + \sigma(x,y) - f_x^{k-1}(x) + \delta = \sigma(x,y)$.
  Thus $\bar f_x^k$ is a radius function.  Define $f_x^k$ as any minimal
  radius function that is dominated by $\bar f_x^k$.

  Since
  $\bar f_x^k(y) = \sigma(x,y) - f_x^{k-1}(x) + \delta = \sigma(x,y) -
  \bar f_x^k(x)$
  for some $y \in X$, we must have
  $f_x^k(x) = \bar f_x^k(x) = f_x^{k-1}(x) - \delta$.  Thus
  $|f_x^{k-1} - f_x^k|_{\infty} \ge \delta$ and
  \[ d_{E(X)}(f_x^{M_x}, x) = f_x^{M_x}(x) = f(x) - M_x\delta =
  f(x) - \Bigl\lfloor\frac{f(x)}{\delta}\Bigr\rfloor\delta < \delta \]
  so we have $d_{E(X)}(f_x^{M_x}, x) < \delta$. On the other hand, by
  Lemma~\ref{lem:dom_distance}, we have
  $|f_x^{k-1} - f_x^k|_{\infty} \le |f_x^{k-1} - \bar f_x^k|_{\infty}
  \le \delta$
  and so
  $d_{E(X)}(f_x^{k-1}, f_x^k) = |f_x^{k-1} - f_x^k|_{\infty} =
  \delta$.  Therefore,
  \begin{align*}
    d_{E(X)}(f, x) 
    & = f(x) \\
    & = M_x\delta + f(x) - M_x\delta \\
    &= M_x\delta + d_{E(X)}(f_x^{M_x}, x) \\
    &= \sum_{k=1}^{M_x} d_{E(X)}(f_x^{k-1}, f_x^k) +
      d_{E(X)}(f_x^{M_x}, x)
  \end{align*}
  where $f_x^0 = f$.  Then, by the triangle inequality, Property~(\ref{itm:isom}) is satisfied.

  To verify Property~(\ref{itm:ineq}), let $y \in X$.  We have
  \[ f(y) + k \delta - \ell_{x,y} = \sigma(x,y) + k \delta - f(x) =
    \sigma(x,y) - f_x^k(x) \le f_x^k(y) \] so the lower bound holds.
  The upper bound on $f_x^k(y)$ given by Property~(\ref{itm:ineq}) is
  $R_k = f(y) + \max\{0,k\delta - \ell_{x,y}\}$.  Suppose Property~(\ref{itm:ineq}) doesn't hold and let $k$ be the least integer for
  which $f_x^k(y) > R_k$.  Then $k > 0$ and $k$ must satisfy
  $f_x^k(y) - f_x^{k-1}(y) > R_k - R_{k-1} \ge 0$.  By the
  construction of $f_x^k$, the fact that $f_x^k(y) > f_x^{k-1}(y)$
  implies that $f_x^{k-1}(y) + f_x^{k-1}(x) - \delta < \sigma(x,y)$
  and that
  $\bar f_x^k(y) = \sigma(x,y) - f_x^{k-1}(x) + \delta = \sigma(x,y) -
  \bar f_x^k(x)$.  Then we must have
  \begin{align*}
    f_x^k(y)
    &= \bar f_x^k(y) \\
    &= \sigma(x,y) - \bar f_x^k(x) \\
    &= \sigma(x,y) - f_x^k(x) \\
    &= f(y) + k \delta - \ell_{x,y} \\
    &\le R_k
  \end{align*}
  which contradicts $f_x^k(y) > R_k$.  Thus we have verified Property~(\ref{itm:ineq}).
  $\hfill\diamondsuit$

  We will now use the sequence $(f_x^k)_k$ of minimal radius functions
  to prove the theorem.  Assume that $a, a' \in S$ satisfy
  $d_S(a,a') \ge \frac{2(K^2-1)}{K^2}|S| + \frac{4\delta + 10C}{K}$.  Such $a$ and $a'$ exist when $K$ is close enough to $1$.  Take $b \in S$ for which $\ell_{\phi(a),\phi(b)} < \delta$.
  Then $d_S(a,a') + d_S(a',b) + d_S(b,a) \le |S|$ so we have
  \begin{align*}
    d_S(a',b)
    &\le |S| - d_S(a,b) - d_S(a,a') \\
    &\le |S| - \frac{2-K^2}{2K^2}|S| + \frac{\ell_{\phi(a),\phi(b)} +
      5C}{K} - d_S(a,a') \\
    &< |S| - \frac{2-K^2}{2K^2}|S| + \frac{5C + \delta}{K} - d_S(a,a') \\
    &\le |S| - \frac{2-K^2}{2K^2}|S| + \frac{5C + \delta}{K} -
      \frac{2(K^2-1)}{K^2}|S| - \frac{4\delta + 10C}{K} \\
    &= \frac{2-K^2}{2K^2}|S| - \frac{3\delta + 5C}{K}
  \end{align*}
  and so
  \[ \frac{2-K^2}{2K^2}|S| - \frac{\ell_{\phi(a'),\phi(b)} + 5C}{K}
    \le d_S(a',b) < \frac{2-K^2}{2K^2}|S| - \frac{3\delta + 5C}{K} \]
  which implies $\ell_{\phi(a'),\phi(b)} > 3\delta$.  So we
  have
  \begin{align*}
    f_{\phi(a')}^3\bigl(\phi(b)\bigr)
    &\le f\bigl(\phi(b)\bigr) +
      \max\bigl\{0, 3\delta - \ell_{\phi(a'),\phi(b)}\bigr\} \\
    &= f\bigl(\phi(b)\bigr) \\
    &\le f_{\phi(a)}^3\bigl(\phi(b)\bigr) - 3\delta + \ell_{\phi(a),\phi(b)} \\
    &< f_{\phi(a)}^3\bigl(\phi(b)\bigr) - 2\delta
  \end{align*}
  where the inequalities are applications of property
  (\ref{itm:ineq}).  Thus
  \[ d_{E(X)}(f_{\phi(a')}^3, f_{\phi(a)}^3) > 2\delta \] and so
  $r(f_{\phi(a')}^3)$ and $r(f_{\phi(a)}^3)$ are distinct elements of
  the metric ball $B\bigl(r(f),5\delta\bigr)$ of radius $5\delta$
  centered at $r(f)$ in $X$.  So, if $\{a_i\}_{i=1}^N$ are
  points of $S$ that subdivide $S$ into segments of length at least
  $\frac{2(K^2-1)}{K^2}|S| + \frac{4\delta + 10C}{K}$ then
  $B\bigl(r(f), 5\delta\bigr)$ contains at least $N$ points.
  Subdividing $S$ evenly we can achieve
  $N = \Bigl\lfloor \Bigl(\frac{2(K^2-1)}{K^2} + \frac{4\delta +
    10C}{K|S|} \Bigr)^{-1} \Bigr\rfloor$.  So we have shown that if
  $X$ admits a $(K,C)$--quasi-isometric embedding of a Riemannian
  circle $S$ and $K$ is close enough to $1$ then for some $x \in X$ we have
  $\bigl|B(x,5\delta)\bigr| \ge \Bigl\lfloor
  \Bigl(\frac{2(K^2-1)}{K^2} + \frac{4\delta + 10C}{K|S|} \Bigr)^{-1}
  \Bigr\rfloor$.
  
  To complete the proof, suppose $X$ is not strongly shortcut.
  Then, for each $K > 1$, there exists $C_K > 0$ and a sequence
  $(\phi_n \colon S_n \to X)_n$ of $(K, C_K)$--quasi-isometric
  embeddings of Riemannian circles where $|S_n| \ge n$.
  The argument above shows that, for each small enough $K > 1$ and each $n \in \N$ there
  exists $x_{K,n} \in X$ satisfying
  $\bigl|B(x_{K,n}, 5\delta)\bigr| \ge \Bigl\lfloor
  \Bigl(\frac{2(K^2-1)}{K^2} + \frac{4\delta + 10C_k}{K|S_n|}
  \Bigr)^{-1} \Bigr\rfloor$.  The expression
  $\Bigl(\frac{2(K^2-1)}{K^2} + \frac{4\delta + 10C_K}{K|S_n|}
  \Bigr)^{-1}$ tends to $\frac{K^2}{2(K^2-1)}$ as $n$ tends to
  infinity so if $n_K \in \N$ is large enough then
  $\bigl|B(x_{K,n_K},5\delta)\bigr| \ge \frac{K^2}{2(K^2-1)} - 1$.
  But this contradicts the uniform bounded geometry assumption on
  $X$ since $\frac{K^2}{2(K^2-1)}$ tends to infinity as $K$ tends
  to $1$.
\end{proof}

\small
\bibliography{bibli}
\bibliographystyle{alpha}

\end{document}